\documentclass[a4paper, 10pt, final]{article} 
\usepackage{a4wide}
\usepackage[notcite,notref]{showkeys}
\usepackage{graphicx}
\usepackage{hyperref}
\usepackage{amsfonts,amsthm,amsmath,amssymb, mathtools}
\usepackage{cleveref}
\usepackage{soul}
\usepackage{tikz}
\usepackage{amsmath}
\usepackage{pgfplots}
\pgfplotsset{compat=1.18}
\usetikzlibrary{intersections,arrows,backgrounds}
\usepackage{pgf}
\usepackage{enumitem}
\newtheorem{example}{Example}[section]
\newtheorem{definition}{Definition}[section]
\newtheorem{theorem}{Theorem}[section]
\newtheorem{proposition}{Proposition}[section]
\newtheorem{lemma}{Lemma}[section] 

\newtheorem{remark}{Remark}[section]

\allowdisplaybreaks \numberwithin{equation}{section}

\newcommand{\R}{\mathbb{R}} \newcommand{\N}{\mathbb{N}}

\newcommand{\abs}[1]{\left|#1\right|}

 \newcommand{\pt}{\partial_t}
 \newcommand{\px}{\partial_x }

\newcommand{\LL}[1]{\mathbf{L^{#1}}}

\newcommand{\BV}{\mathbf{BV}}

\newcommand{\CC}[1]{\mathbf{C^{#1}}}
\newcommand{\CCc}[1]{\mathbf{C^{#1}_c}}

\DeclareMathOperator{\tv}{TV}
\DeclareMathOperator{\os}{Osc}

{%

  \begin{enumerate}}%
  {\end{enumerate}}

{%

  \begin{enumerate}}%
  {\end{enumerate}}

\begin{document}

\title
{On the structure of optimal solutions\\ of conservation laws at a junction\\ with one incoming and one outgoing arc}

\author{F.~Ancona$^1$ \and A.~Cesaroni$^2$ \and G. M. Coclite$^3$
  \and M. Garavello$^4$}

\maketitle

\footnotetext[1]{Dipartimento di Matematica ``Tullio Levi-Civita'',
  Università di Padova, Via Trieste 63, 35121 Padova, Italy.\hfill\\
  Email address: \texttt{ancona@math.unipd.it}
  URL: \texttt{http://www.math.unipd.it/{\~{}}ancona/}}

\footnotetext[2]{Dipartimento di Matematica ``Tullio Levi-Civita'',
  Università di Padova, Via Trieste 63, 35121 Padova, Italy.\hfill\\
  Email address: \texttt{annalisa.cesaroni@unipd.it}
  URL: \texttt{http://www.math.unipd.it/{\~{}}acesar/}}

\footnotetext[3]{Dipartimento di Meccanica, Matematica e Management,
  Politecnico di Bari, Via E.~Orabona 4,  I--70125 Bari, Italy.\hfill\\
  Email address: \texttt{giuseppemaria.coclite@poliba.it}
  URL: \texttt{https://sites.google.com/site/coclitegm/}}

\footnotetext[4]{Dipartimento di Matematica e Applicazioni,
  Università di Milano Bicocca,
  Via  R. Cozzi 55, I--20125 Milano, Italy.\hfill\\
  Email address: \texttt{mauro.garavello@unimib.it}
  URL: \texttt{https://sites.google.com/site/maurogaravello/}}


 \begin{abstract}
 We consider a min-max problem 
for strictly concave conservation laws on a 1-1 network, with inflow controls acting at the junction. 
We investigate the minimization problem for a functional measuring the total variation of the flow of the solutions at the node,
among those solutions that maximize the time integral of the flux.
To formulate this problem we establish a regularity result showing that
the total variation of the boundary-flux of the solution of an initial-boundary value problem 
is controlled by the total variation of the initial datum
and of the flux of the boundary datum. 
In the case the initial datum  is monotone,
we show that the flux 
of the entropy weak solution at the node 
provides 
an optimal inflow control for this min-max problem.
We also exhibit  two prototype examples showing that, in the case where the 
initial datum is not monotone, the flux 
of the entropy weak solution is no more optimal. 

   
   \medskip

   \noindent\textbf{MSC~2020:} 35F25, 35L65, 90B20.
 \end{abstract}

\section{Introduction}

In this note, we consider an optimal control problem for scalar conservation laws evolving on a 1–1 network consisting of a junction (node) with one incoming and one outgoing edge.
We model the 
incoming edge  with the half-line $I_1\doteq (-\infty, 0)$, and
the outgoing one with the half-line $I_2\doteq (0, +\infty)$,
so that the junction is   sitting at  $x=0$.
On each edge $I_i$ the evolution of the state variable $u_i(t,x)$ is governed by a
scalar conservation law 
supplemented with some initial condition $\overline u_i\in  \mathbf{L^\infty}(I_i; \R)$,
$i=1,2$.
This yields to study the Cauchy problems
on the two semilines $I_i$, $i=1,2$:
\begin{equation}
  \label{incoming} 
  \begin{cases} \pt u_1 +\px f(u_1)=0, \qquad & x<0, \, t>0,\\ 
    u_1(0,x)= \overline u_1(x), & x<0,\end{cases}
\end{equation} 
\begin{equation}\label{outgoing}
  \begin{cases} \pt u_2 +\px f(u_2)=0,  \qquad & x>0, \, t>0,\\ 
    u_2(0,x)= \overline u_2(x), & x>0,\end{cases}
\end{equation} 
where we assume that the flux  $f:\R\to\R$
is a twice continuously differentiable, (uniformly) strictly concave map.
The equations~\eqref{incoming}, \eqref{outgoing} are usually coupled through a Kirkhoff type condition at the boundary $x=0$:
\begin{equation}\label{conservation}
f(u_1(t,0))= f(u_2(t,0))
\qquad\textrm{for a.e.} \ \ t>0.
\end{equation}
Here, and throughout the following, we adopt the notation
$u_1(t,0)\doteq \lim_{x\to 0^-} u_1(t,x)$,
$u_2(t,0)\doteq \lim_{x\to 0^+} u_2(t,x)$,
for the one-sided limit 
of $u_i(t,\cdot)$ at $x=0$.
The equation~\eqref{conservation} expresses the conservation of the total mass through the node.
Note that the system~\eqref{incoming}, \eqref{outgoing}, and~\eqref{conservation}, 
admits in general infinitely many entropy admissible solutions.
One may recover the uniqueness of solutions to the
nodal Cauchy problem~\eqref{incoming}, \eqref{outgoing}, \eqref{conservation},
requiring that the flux trace at the junction
\begin{equation}\label{gammat-0} \gamma(t)\doteq f(u_1(t,0)) = f(u_2(t, 0)),
\qquad t>0,\end{equation}
satisfies some further condition. 
Such conditions are expressed in terms of time-dependent germs \cite{akr,cardsolo}, or introducing the notion of flux-limited solutions 
for Hamilton-Jacobi equations~\cite{im} (exploiting the equivalence between the two classes of equations~\cite{card1}), 
or are obtained via vanishing viscosity approximations~\cite{cg}.
Here, we follow the different perspective introduced in~\cite{accg}
to analyze control problems for conservation laws
models at a junction with $n$ incoming and $m$ outgoing edges. Namely,   
we regard the flux trace $\gamma$ in~\eqref{gammat-0} as an {\it inflow control}
acting at the junction.
More precisely, for a fixed $T>0$, 
we consider
$\gamma\in \mathbf{L^\infty}((0,T);\R)$
as an {\it admissible inflow control}
if there exists a pair of boundary data 
$k_1, k_2 \in \mathbf{L^\infty}\left((0,T); \R \right)$,
such that, letting $u_1:[0,T]\times (-\infty,0]\to\R$, $u_2:[0,T]\times [0,+\infty)\to\R$, be the entropy admissible weak solutions of the mixed initial boundary value problems
\begin{equation}
  \label{bdry-incoming} 
  \begin{cases} \pt u_1 +\px f(u_1)=0, \qquad & x<0, \, t>0,\\ 
    u_1(0,x)= \overline u_1(x), & x<0,\\
    u_1(t,0)= k_1(t), & t>0,\end{cases}
\end{equation} 
\begin{equation}
  \label{bdry-outgoing} 
  \begin{cases} \pt u_2 +\px f(u_2)=0, \qquad & x>0, \, t>0,\\ 
    u_2(0,x)= \overline u_2(x), & x>0,\\
    u_2(t,0)= k_2(t), & t>0,\end{cases}
\end{equation} 
the corresponding fluxes satisfy the  
constraint~\eqref{gammat-0}.
The Dirichlet boundary conditions in~\eqref{bdry-incoming}, \eqref{bdry-outgoing}
are 
interpreted in the relaxed sense 
of Bardos, Le Roux, N\'ed\'elec~\cite{BLN} which, for strictly concave fluxes, 
are equivalent to the conditions (iii)
stated in Definition~\ref{def-ent-sol-ibvp} 
(in Section~\ref{sec:prelim}) of entropy admissible solution of the mixed initial-boundary
value problem (cfr.~\cite{lf}).
Given a fixed initial datum
\begin{equation}
\label{indatum}
    \overline u(x)=
    \begin{cases}
        \overline u_1(x),\quad&\text{if}\quad x<0,
        \\
        \noalign{\smallskip}
        \overline u_2(x),\quad&\text{if}\quad x>0,
    \end{cases}
\end{equation}
for any admissible inflow control 
$\gamma\in \mathbf{L^\infty}((0,T);\R)$,
we will denote by 
\begin{equation}\label{ugamma} 
  u(t,x; \gamma)\doteq \begin{cases}u_1(t,x),\quad&\text{if}\quad x<0, \\ 
\noalign{\smallskip}
u_2(t,x), \quad&\text{if}\quad x>0,\end{cases}
\end{equation} 
the unique solution of the nodal Cauchy problem 
~\eqref{incoming}, \eqref{outgoing}, \eqref{conservation}, \eqref{gammat-0}.

Optimization issues for conservation laws evolving on networks has attracted 
an increasing attention in the last two decades (see the monographs~\cite{bdggp,bcghp,g-h-p} and the references therein). 
The interest on these problems is motivated by a wide range of applications in  different areas
such as data flow and telecommunication~\cite{cmpr}, 
gas or water pipelines~\cite{cghs,hlmss}, production 
processes~\cite{Da-G-H-P, Da-M-H-P, hfy, LM-A-H-R}, 
biological resources~\cite{CoGa}, and 
blood circulation~\cite{ca-d-d-m}.
In particular, optimal control problems for conservation laws with concave flux 
have found a natural application in traffic flow models, firstly introduced in the seminal papers by Lighthill, Whitham and 
Richards (LWR model)~\cite{li-whi,ri}.
Several optimization problems related to different
traffic performance indexes have been considered in the literature 
to  improve the mean travel time in
a given road segment~\cite{am1,am2,cdpr1,goatin-piacentini,tk},  
to diminish
the queue length due to a red light or stop sign~\cite{colombo-rossi},
to minimise stop-and-go waves~\cite{cgr,col-gro,
tsmh}, 
or to get a fuel consumption reduction~\cite{ac, ramadan}.

Within the above setup,
in the general case of a junction with $n$ incoming and $m$ outgoing edges  
we have established the existence of optimal inflow controls 
for
various functionals depending on the 
value of the solution at
the node~\cite{accg},
and along the incoming or outgoing 
edges~\cite{accg2}. 

The object of this note is to derive a characterization of optimal inflow controls for a class
of optimal control problems on  a 1–1 network.
Namely, we first consider the problem of maximizing the total flux at the junction
\begin{equation}\label{max-J-0}
  \sup_{\gamma \in \mathcal{U}} \int_0^T 
  f(u_1(t, 0;\gamma)) dt
  = \sup_{\gamma \in {\mathcal{U}}} \int_0^T   \gamma(t) \  dt,
\end{equation}
where $\mathcal{U}$ denotes the class of admissible inflow controls.
Our first result (Theorem~\ref{maxtheo1}) shows that, letting $\widetilde u :[0,T]\times \R\to\R$
denote the entropy weak solution of the Cauchy problem
\begin{equation}
  \label{eq:CP-classic}
\begin{cases} \pt u + \px f(u) = 0, \quad & x \in \R, \, t > 0,\\ 
    u(0,x)= \overline u(x), & x \in \R,\end{cases}
\end{equation}
the corresponding flux $\widetilde\gamma\doteq f\circ \widetilde u(\cdot, 0)$ provides a maximizer of~\eqref{max-J-0}.
Since the optimization problem~\eqref{max-J-0} admits in general more than one maximizer, we next address a
minimization problem for a functional 
$\mathcal F$ measuring the total variation of the solution at the node (see definition~\eqref{tv-def})
\begin{equation}\label{maxmin-0}
  \inf_{\gamma \in \mathcal{U}_{max}} \mathcal F\big(u(\cdot,\cdot\,;\gamma)\big),
\end{equation} 
where $\mathcal{U}_{max}$ denotes the set of the maximizers for~\eqref{max-J-0}. 
The goal of this min-max problem is to keep as small as possible the oscillations of the 
flux of the solution at the node among those solutions that maximize the time integral of the  flux. 
To formulate this
minimization problem
we need first to ensure 
that the flux of the solutions of the initial boundary value problems~\eqref{bdry-incoming}, \eqref{bdry-outgoing}, 
evaluated at the boundary $x=0$, is a function of bounded variation ($\mathbf{BV}$), 
provided that the initial data $\overline{u}_i$,
and the
flux of the boundary data $f\circ k_i$,  are $\BV$ functions as well. 
This regularity result 
is established in Theorem~\ref{bvfluxbdrytrace}
of the Appendix.
 
The main results of this note (Theorem~\ref{mon}, Theorem~\ref{mon2}) show that,
in the case where the initial datum $\overline u$ is monotone,
the flux $\widetilde\gamma$
of the entropy weak solution of~\eqref{eq:CP-classic} 
provides 
an optimal inflow control for the min-max problem~\eqref{maxmin-0}.
On the other hand, in Section~\ref{sectionesempi} we discuss  two prototype examples showing that, in the case where the 
initial datum $\overline u$ is not monotone, the flux~$\widetilde\gamma$ is no more optimal. 
This leads us to formulate the following:
\begin{itemize}
[leftmargin=14pt]
    \item[] 
    {\it Conjecture: if the flux $\widetilde \gamma$ of the  entropy weak solution of~\eqref{eq:CP-classic} is not monotone on $(0,T)$, then an optimal control 
    for~\eqref{maxmin-0}
    must be a piecewise constant flux of a non-entropic solution of~\eqref{eq:CP-classic}.
    Moreover, such a solution admits a shock discontinuity emerging tangentially from the junction $x=0$.
}
\end{itemize}

From the perspective of traffic model applications, our result is quite natural. 
Specifically, when the traffic on both incoming and outgoing roads exhibits homogeneous features 
(meaning it is either in acceleration or in  deceleration regime),   
no control is needed at the road intersection 
in order to enhance the overall performance. However, when the regime 
on the roads is heterogeneous, so presenting both deceleration and acceleration areas, 
then implementing a different condition at the road intersection can significantly improve car circulation.

The analysis of the optimization problems considered in this note relies 
on the method of generalized characteristics
developed by Dafermos~\cite{d,daf-book} for conservation laws with convex or concave fluxes, 
and applied to the present setting of
a 1-1 network.
Here we have restricted  the study of these optimization problems to the case 
where the dynamics along the incoming and outgoing edges is described by the
same conservation law. 
As a natural next step, we plan to investigate the structure of  solutions for this type of optimal control 
problems when the flux of the incoming and outgoing edges is different
(in order to describe traffic flow dynamics on roads of varying amplitudes or surface conditions). In this different scenario 
we will exploit the analysis pursued in~\cite{at1,at2}
on the structure of solutions to conservation laws with discontinuous flux. We expect 
to derive conditions (on some classes of initial data) which ensure 
 the optimality of inflow controls associated to the so-called {\it interface connections} $(A,B)$.

Recent results on the structure of optimal solutions for another class of control 
problems on a 1-1 network
were obtained in~\cite{card1}
exploiting the equivalence between conservation laws and Hamilton Jacobi equations.

The note is organized as follows.
In Section~\ref{sec:prelim} we recall some preliminary results on initial-boundary value problems, 
present the general control setting, 
and state the main results of the note. 
In Section~\ref{sec:entrsolmax} we prove the optimality for~\eqref{max-J-0} of the flux of the entropy weak solution of~\eqref{eq:CP-classic}. In Section~\ref{sectionesempi} we propose two examples of non-monotone initial data $\overline u$
for which the entropy weak solution of~\eqref{eq:CP-classic} is not an optimal solution of~\eqref{maxmin-0}. 
In Section~\ref{entrsolvminmax} we establish the main results of the note on the optimality for~\eqref{maxmin-0} of 
the flux of the entropy weak solution of~\eqref{eq:CP-classic} when the initial datum is monotone. Finally, 
in the Appendix~\ref{appendix} we provide a proof of the $\mathbf{BV}$-regularity of the trace of the flux of 
the solution of an initial-boundary value problem.

\section{Preliminaries and main results}
\label{sec:prelim}
In traffic flow models, usually the traffic density takes values in a compact domain $\Omega = [0, u^{\max}]$,
where $u^{\max}$  denotes the 
maximum possible density (bump-to-bump) inside the road. 
The fundamental diagram $f: \Omega \to\R$ has the expression $f(u)=u v(u)$ where $v$
is the average speed of vehicles. 
In most models the velocity $v$ 
satisfies $v'(u)\leq \overline v<0$
for all $u$,
and $u\to u v'(u)$
is a non-increasing map.
Because of the finite propagation speed,
it is not restrictive to assume
that the domain of~$f$ is the whole real line.
Thus, throughout the note we make the 
general assumption
that the  flux function 
is a twice continuously differentiable
map $f:\R\to\R$, which is (uniformly) strictly concave, and that attains its global maximum at the point $\theta$, 
i.e. \begin{equation}\label{def-theta} f(\theta)=\max_{u\in\R} f(u).\end{equation}

We recall that a function $g: (a,b)\to \R$, $a,b\in \R\cup\{-\infty,+\infty\}$, of bounded variation ($\mathbf{BV}$) 
admits  left limits at any point~$t\in (a,b]$, 
and right limits at any point $t\in[a,b)$.
We will denote them with $g(t^-)$
and $g(t^+)$, respectively.
We  
define $\tv_{(a, b)}\{g\}$ as the 
{\it essential variation} of $g$ on the  interval $(a,b)$ 
which coincides with the {\it pointwise variation} on $(a,b)$
of the right continuous or left continuous representative of~$g$
(see~\cite[Section 3.2]{AFP}). 
More explicitly for the right continuous representative $g$ we write:    
\begin{equation}
  \label{tv}
  \tv_{(a,b)}\{g\}\doteq \sup_{a< t_0<t_1<\cdots<t_N<  b}
  \left\{\sum_{\ell=1}^N |g(t_\ell) - g(t_{\ell-1})|\right\}.
\end{equation}
We will sometimes also use the notation, for $(c,d)\subset (a,b)$, 
\begin{equation}
  \label{tvchiusa}
  \tv_{[c,d]}\{g\}\doteq  \tv_{(c,d)}(g)+ |g(c^+)-g(c^-)|+|g(d^+)-g(d^-)|.
\end{equation}


\subsection{Admissible junction controls}
\label{sec:bvp}
Following \cite{accg}, we introduce in this section the set of admissible junction inflow controls
associated to pairs of solutions of the initial boundary value problems~\eqref{bdry-incoming}, \eqref{bdry-outgoing}.
We recall  (see \cite[Remark~2.3]{accg}) that  an entropy weak solution $u_1(t,\,\cdot)$
of~\eqref{bdry-incoming} ($u_2(t,\,\cdot)$ of~\eqref{bdry-outgoing}) satisfies a-priori
BV-bounds (due to Ole\v{\i}nik-type inequalities) and  admits the one-sided limit $u_1(t,0^-)$ ($u_2(t,0^+)$) 
at any fixed time $t>0$.
As observed in the introduction, 
throughtout the paper we will adopt the notation $u_1(t,0)\doteq u_1(t,0^-)$, $u_2(t,0)\doteq u_2(t,0^+)$.
Moreover, we will use the notations
$I_1\doteq (-\infty, 0)$, $I_2\doteq (0, +\infty)$, $J_1\doteq [\theta,+\infty)$, $J_2\doteq (-\infty,\theta]$.
The Dirichlet boundary conditions at $x=0$ must be
interpreted in the relaxed sense 
of Bardos, Le Roux, N\'ed\'elec~\cite{BLN}.
Therefore we shall adopt the following
definition of entropy admissible weak solution of the mixed initial-boundary
value problem~\eqref{bdry-incoming}, \eqref{bdry-outgoing} (cfr.~\cite{lf})
\begin{definition}
  \label{def-ent-sol-ibvp}
  Given $\overline u_i\in  \mathbf{L^\infty} (I_i;  \R)$ and
  $k_i\in \mathbf{L^\infty} \big((0,T);  J_i\big)$, $T>0$,
  we say that a function
  $u_i\in\mathbf{C}\big([0,T];\, \mathbf{L}^1_{{\rm loc}}
  (I_i;  \R)\big)$
  is an entropy admissible weak solution
  of~\eqref{bdry-incoming} (resp. of~\eqref{bdry-outgoing}) if :
  \begin{enumerate}[label=(\roman*) , align=left]
  \item For every $t > 0$,
    $u_i(t) \in \mathbf{BV_{\rm loc}}\left(I_i; \R\right)$ and admits the one-sided limit at $x=0$.

  \item $u_i$ is a weak entropic solution
    of~\eqref{incoming} (resp. of~\eqref{outgoing}) with initial data $\overline u_i$, i.e. for all 
    $\varphi\in\CCc{1}\big([0,T)\times I_i;  \R\big)$ there holds
    \begin{equation*}
      \int_0^T\!\!\int_{I_i}\Big[u_i\,  \pt\varphi + f(u_i)\, 
      \px\varphi \Big](t,x)~dxdt + \int_{I_i}\overline u_i (x)
      \varphi(0,x)dx\, = 0\,,
    \end{equation*}
    and the Lax entropy condition  is satisfied 
    \begin{equation*}
      u_i(t,x^-)\leq u_i(t,x^+),\qquad\quad
      t>0, \ x \in I_i.
    \end{equation*}
  \item\label{def:solution:bc}
    The boundary condition in~\eqref{bdry-incoming} is verified if,
    for a.e. $t>0$, there holds:
%
\begin{equation}
\label{Dir-bc-in}
  \begin{array}{c@{\qquad}c}
    \textrm{either} &  u_1(t,0)=k_1(t),
    \vspace{.1cm}
    \\
    \noalign{\smallskip}
    \textrm{or}  & \quad f'(u_1(t,0))> 0  \   \
    \textrm{ and } \ f(u_1(t,0))\leq 
    f\big(k_1(t)\big);
  \end{array}
\end{equation} 
the boundary condition in~\eqref{bdry-outgoing} is verified if,
for a.e. $t>0$, there holds:
\begin{equation*}
  \begin{array}{c@{\qquad}c}
    \textrm{either} &  u_2(t,0)=k_2(t),
    \vspace{.1cm}
    \\
    \noalign{\smallskip}
    \textrm{or}  &  \quad f'(u_2(t,0))< 0  \   \
    \textrm{ and } \ f(u_2(t,0))\leq 
    f\big(k_2(t)\big).
  \end{array}
\end{equation*} 
  \end{enumerate}
\end{definition}
 
\begin{remark}\upshape
  In Definition~\ref{def-ent-sol-ibvp}
  we are considering only boundary data that have characteristics entering the domain, 
  that is  $k_1 \in \mathbf{L^\infty}\left((0,T); [\theta,+\infty) \right)$ 
  and $k_2 \in \mathbf{L^\infty}\left((0,T); (-\infty, \theta] \right)$.
  As observed in \cite{lf}, this is not a restrictive assumption. 
  Indeed, if one considers more general boundary data $k_i \in \mathbf{L^\infty}\left((0,T);\R\right)$,
  one can recover the same
  entropy admissible weak solutions $u_i(t,x)$, $i=1,2$, 
  of \eqref{bdry-incoming}, \eqref{bdry-outgoing},
  replacing the boundary data $k_i$ with the normalized boundary data
  $\widetilde k_i$  defined by
  \begin{equation}
    \label{normal-bdry-data25}
    \widetilde k_i(t)\doteq 
    \begin{cases}
      f_+^{-1}\big(f(u_1(t,0))\big), & \textrm{if} \quad\ i=1,
      \\
      \noalign{\medskip}
      f_-^{-1}\big(f(u_2(t,0)\big), & \textrm{if} \quad\ i=2,
    \end{cases}
  \end{equation}
  where we denote as $f_-, f_+$ the restrictions of $f$ to the intervals $(-\infty, \theta]$
  and $[\theta, +\infty]$, respectively.
\end{remark}
Now we are in the position to formulate the definition of admissible inflow control.
\begin{definition}\label{defiammissibile} 
  Given $T>0$, $\overline u\in 
  \mathbf{L^\infty}(\R;\R)$,
  as in~\eqref{indatum}, with
  $\overline u_i\in  \mathbf{L^\infty}(I_i; \R)$, $i=1,2$,
  we say that $\gamma \in \mathbf{L^\infty}((0,T);\R)$
  is an admissible junction inflow
  control if there exist boundary data\linebreak
  $k_1 \in \mathbf{L^\infty}\left((0,T); [\theta,+\infty) \right)$, $k_2 \in \mathbf{L^\infty}\left((0,T); 
  (-\infty,\theta] \right)$ 
  such that, for a.e. $t\in [0,T]$,
  there holds
  \begin{equation}
    \label{gammat} 
    \gamma(t)=f(u_1(t,0)) = f(u_2(t, 0)),
  \end{equation}
  where $u_1$ is the entropy weak solution to~\eqref{bdry-incoming} with initial 
  datum $\overline u_1$ and boundary datum $k_1$,  while $u_2$ is the entropy weak solution to~\eqref{bdry-outgoing} 
  with initial datum $\overline u_2$ and boundary datum $k_2$. 

  Let us denote  
  \begin{equation}
    \label{ugamma-2} 
    u(t,x; \gamma)= \begin{cases}u_1(t,x),
    & x<0 \\ u_2(t,x),
    & x>0 \end{cases}
  \end{equation} 
  so that   
  $\gamma (t)= f(u(t, 0; \gamma)) $ for a.e. $t$.
\end{definition}

\begin{remark}
\label{rem:entr-sol}
Observe that, if $\widetilde u : (0,T)\times\R\to \R$ denotes the 
entropy admissible weak solution
to~\eqref{eq:CP-classic}, with initial datum
$\overline u\in 
  \mathbf{L^\infty}(\R;\R)$
  as in~\eqref{indatum}, 
  $\overline u_i\in  \mathbf{L^\infty}(I_i; \R)$ with $i=1, 2$,
then the flux $f\circ \widetilde u(\cdot\,;0)$ is an admissible junction inflow control.
In fact, one can immediately check that
letting $\widetilde u_1$, $\widetilde u_2$ denote the restriction of $\widetilde u$ to $\{x<0\}$, $\{x>0\}$,
respectively, then we have:
\begin{itemize}
[leftmargin=14pt]
    \item[-]
    $\widetilde u_1$ is the entropy weak solution to~\eqref{bdry-incoming} with initial datum $\overline u_1$ and boundary datum
    \begin{equation*}
        k_1(t)\doteq f_+^{-1}\big(f(\widetilde u_1(t,0))\big),
        \quad t\in (0,T);
    \end{equation*}
    \item[-]
    $\widetilde u_2$ is the entropy weak solution to~\eqref{bdry-outgoing} with initial datum $\overline u_2$ and boundary datum
    \begin{equation*}
        k_2(t)\doteq f_-^{-1}\big(f(\widetilde u_2(t,0))\big),
        \quad t\in (0,T);
    \end{equation*}
    \item[-] for a.e. $t \in (0, T)$,
    \begin{equation}
    f(\widetilde u_1(t,0))=
   f(\widetilde u_2(t, 0)). 
  \end{equation}
    \end{itemize}
Then, it follows that the map $\widetilde \gamma(t)\doteq f(\widetilde u_1(t,0))$, $t\in(0,T)$, 
satisfies the conditions of Definition~\ref{defiammissibile}.
Therefore, we have 
\begin{equation*}
    \widetilde u(t,x) = u(t,x;\widetilde\gamma\,)
    \qquad\forall~(t,x)\in (0,T)\times\R\,,
\end{equation*}
with $u(t,x;\widetilde\gamma\,)$ defined as in~\eqref{ugamma-2} replacing $u_i$ with $\widetilde u_i$.
\end{remark}

\subsection{The control problem}

For every fixed $T>0$, given $\overline u\in 
  \mathbf{L^\infty}(\R;\R)$,
we introduce the set of admissible controls according to Definition~\ref{defiammissibile}:
\begin{equation}\label{controlli}
 \mathcal{U}\doteq  \mathcal{U}(\overline u)\doteq   \Big\{ \gamma\in  \mathbf{L^\infty}
 \big([0,T];\, \R
 \big)
\ \big| \    \gamma\ \textrm{is an   admissible junction inflow  control}     
    \Big\}, \end{equation} 
and we  consider the  maximization problem
\begin{equation}
  \label{max-J}
  \sup_{\gamma \in {\mathcal{U}}} \int_0^T 
  f(u_1(t, 0;\gamma))\, dt 
  = \sup_{\gamma \in {\mathcal{U}}} \int_0^T   \gamma(t) \,  dt.
\end{equation}
Our first result is the following:
\begin{theorem} 
  \label{maxtheo1} 
  Given $T>0$ and $\overline u\in \mathbf{L^\infty}(\R;\R)$, 
  let $\widetilde u$ be the entropy admissible weak solution
  to the Cauchy problem~\eqref{eq:CP-classic}. Then the composite function  
  $\widetilde\gamma\doteq f\circ \widetilde u (\cdot\,0)$ solves the maximization problem \eqref{max-J},
  i.e. there holds   
  \begin{equation*}
    \int_0^T f(\widetilde u(t,0)) \, dt =
    \max_{\gamma \in {\mathcal{U}}} \int_0^T \gamma(t) \, dt.
  \end{equation*}
\end{theorem} 
The proof of   this result is given in 
Section~\ref{sec:entrsolmax}.
One can easily verify that the optimization problem~\eqref{max-J} admits other maximizers different from $\widetilde\gamma$, 
which are
junction inflow controls that present
rarefaction waves or shocks generated at the node $x=0$ at positive times (see Section~\ref{sectionesempi}). 
Since an usual objective in traffic management 
is to keep the traffic flow as fluent as possible, 
it is then natural 
to investigate the existence of maximizers
of~\eqref{max-J} that minimise a cost functional
measuring the total variation of the corresponding solution at the node $x=0$
(e.g. see~\cite{accg,cgr}).

Therefore for every fixed $T>0$, given $\overline u\in 
  \mathbf{L}^\infty(\R;\R)$,
we restrict   the set  of maximizers of \eqref{max-J} to the  set  of inflow controls: 
\begin{equation}
    \label{eq:admiscontolset-max}
    \mathcal{U}_{max} 
  \doteq \mathcal{U}_{max}(\overline u)
  \doteq \left\{ \bar \gamma\in \mathcal{U}\cap 
  \mathbf{BV}   
 \big([0,T];\, \R \big)\ \bigg|\ \int_0^T f(u_1(t, 0;\bar\gamma)) dt
  = \max_{\gamma \in {\mathcal{U}}} \int_0^T \ f(u_1(t, 0;\gamma)) dt\right\}
\end{equation}
which are admissible for the functional we are going to minimise. For this reason we will assume  $\overline u\in 
  \mathbf{BV}(\R;\R)$, so to guarantee that this set is not empty, see the following Remark. 



%

\begin{remark}\upshape 
\label{rem:entr-sol-bv}
If $\overline u\in 
  \mathbf{BV}(\R; \R)$,
  and $\widetilde u$ is the 
entropy admissible weak solution
to~(\ref{eq:CP-classic}),
  then the associated  boundary inflow control
  $\widetilde\gamma\doteq f\circ \widetilde u (\cdot, 0)$ is a function of bounded variation on $(0,T)$ 
  as a consequence of a recent result obtained   in \cite[Theorem 1.1]{marconi1}. 
  More precisely the result   is the following: for every $T>0$ and for every $x\in\R$,     
  the map $t\to f(\widetilde u(t , x))$ is in 
 $ \mathbf{BV}([0,T]; \R)$, with 
 \begin{equation}\label{stimabv} 
  \tv_{(0, T)} f(\widetilde u(\cdot\,, x))\leq C(\tv(\overline u), \|f'\circ\overline u\|_{\LL{\infty}}),
 \end{equation}
 where $C(\tv(\overline u), \|f'\circ\overline u\|_{\LL\infty})$ denotes an explicit constant only depending on $\tv(\overline u) $ and on
the Lipschitz constant of $\|f'\circ\overline u\|_{\LL\infty}$.

Thus, by Remark~\ref{rem:entr-sol}, we have
  $\widetilde\gamma\in \mathcal{U}_{max}$ 
and in particular  the set $\mathcal{U}_{max}$ is not empty.
\end{remark}

In  Appendix~\ref{appendix} we complement the  result~\cite[Theorem 1.1]{marconi1} recalled in  Remark \ref{rem:entr-sol-bv} with  a regularity result
for solutions of initial-boundary value problems. 
More precisely in Theorem \ref{bvfluxbdrytrace} 
we show that if  $\overline u_1\in \mathbf{BV} (\R; \R)$,  
$k_1 \in \mathbf{L^\infty}\left((0,T); [\theta,+\infty) \right)$  with 
$f(k_1)\in \mathbf{BV}((0,T); \R)$
then the flow map 
$t\mapsto   f\big(u_1(\cdot, 0)\big)$ of the  entropy admissible weak solution of \eqref{bdry-incoming}
is in $\mathbf{BV}((0,T); \R)$.

This implies that in order to obtain  admissible  boundary inflow control $\gamma\in \mathcal{U}$ 
that has bounded total variation it is sufficient to choose boundary data $k_1, k_2$ 
in~\eqref{bdry-incoming}, \eqref{bdry-outgoing} with $f(k_i)\in \mathbf{BV}((0,T); \R)$.

\begin{remark}\upshape
  Notice that to construct  admissible inflow controls it is not necessary that the boundary data  $k_i$ 
  are in  $\mathbf{BV}((0,T); \R)$, but only that $f(k_i)$ are in $\mathbf{BV}((0,T); \R)$. 
  The requirement $f(k_i)\in\mathbf{BV}((0, T); \R)$ is less restrictive: 
  indeed in general the entropy solution $\widetilde u$ satisfies $f(\widetilde u(\cdot,0))\in \mathbf{BV}((0, T); \R)$,
  but it does not satisfy $\widetilde u(\cdot,0)\in \mathbf{BV}((0, T); \R)$ as shown by an explicit counterexample 
  in \cite{marconi1}.
\end{remark}

Observe that, if $\gamma \in \mathbf{BV}   \big([0,T];\, \R \big)$ is an admissible junction inflow control,
  according to  \Cref{defiammissibile} we get that 
\[\tv_{(0, T)}(f\big(u_1(\cdot\,, 0; \gamma)\big))=\tv_{(0, T)}(f(u_2(\cdot\,, 0;\gamma))),   \] 
 where the total variation is defined in \eqref{tv}. 
We then introduce the following functional: for 
any    $\gamma\in \mathcal{U}_{max}$, let 
\begin{align}
  \label{tv-def}
  \mathcal F_{[0, T]}(\gamma)
  & \doteq \tv_{(0, T)}(f(u_1(\cdot, 0;\gamma)))  
  \\ 
  \nonumber &\quad +|f(\overline u(0^-))- \gamma(0^+)|+|f(u_1(T,0; \gamma))
  - \gamma(T^-)|
  \\
  \nonumber &\quad+|f(\overline u(0^+))- \gamma(0^+)|+|f(u_2(T,0; \gamma))- \gamma(T^-)|.
\end{align}
This definition of the  functional $\mathcal{F}$ aims to take into account  also variations of  
the flux at the node~$x=0$ which occur at the beginning and at the end of the interval of time $[0,T]$.

We study the  following minimization problem
\begin{equation}\label{maxmin}
  \inf_{\gamma \in {\mathcal{U}}_{max}} \mathcal F_{[0, T]}(\gamma).
\end{equation}  
Due to Theorem \ref{maxtheo1} and Remark 
 \ref{rem:entr-sol-bv}
we get that \[
  \inf_{\gamma \in {\mathcal{U}}_{max}} \mathcal F_{[0, T]}(\gamma)=
  \inf_{\gamma \in {\mathcal{U}}_{max}^M} \mathcal F_{[0, T]}(\gamma),
\] where \[
 \mathcal{U}^M_{max} 
  = \Big\{  \gamma\in \mathcal{U}_{max}\ \big|\,  \mathcal{F}_{[0,T]} (\gamma)\leq M\doteq\mathcal{F}_{[0,T]} 
  f(\widetilde u(\cdot, 0))\Big\}.
\]
Note that this set is compact  in $\mathbf{L}^1_{loc}$, as proved in \cite{accg}, 
so  we conclude the following existence result for \eqref{maxmin}.
\begin{theorem} 
  Given $T>0$, and $\overline u\in 
  \mathbf{BV}(\R;\R)$, there exists a solution to \eqref{maxmin}. 
\end{theorem} 

In \cite{accg} existence of solutions to \eqref{max-J} and \eqref{maxmin} have been proved in much more generality, 
for networks with $n$ incoming roads and $m$ outgoing roads.

\subsection{Main result} 
The  main result of the note shows that    the  entropy admissible weak solution  solves the minimization problem \eqref{maxmin} when the initial datum $\overline u$ is monotone. This provides an optimal control characterization of the entropic solution for monotone initial data. The precise statement of our result is the following.
\begin{theorem}
Let $\widetilde u$ be the unique entropy  admissible weak solution
to~\eqref{eq:CP-classic}. 
If the initial datum $\overline u$  is monotone, then
\[  \min_{\gamma \in {\mathcal{U}}^M_{max}} \mathcal F_{[0, T]}(\gamma)= \mathcal F_{[0,T]} (\widetilde \gamma),   \] where $\widetilde \gamma= f(\widetilde u(\cdot, 0))$. 
\end{theorem} 
This   Theorem follows directly from \Cref{mon} for monotone non-decreasing initial data and from \Cref{mon2} 
for monotone non-increasing initial data.

On the other hand, without the monotonicity assumption on the initial data, we cannot expect that the entropy 
solution minimizes the total variation functional $\mathcal F_{[0, T]}$. In particular in Section  \ref{sectionesempi}, 
we provide two prototypical examples of non-monotone initial data for which the minimum in \eqref{maxmin} 
is achieved by a   not entropic solution.

\section{The entropy  admissible weak solution  maximizes the flow} 
\label{sec:entrsolmax}
In this section we provide the proof of \Cref{maxtheo1}.  

We start with a lemma, which will be useful in the following. 

\begin{lemma}
  \label{uniq} \  \ 
  \begin{enumerate}
    \item Let $u_1, v_1$ be two entropy admissible  weak solutions to \eqref{bdry-incoming}, 
      with the same  initial datum $\overline u\in \mathbf{L^\infty}
      ((-\infty, 0]; \R)$ and 
      such that $u_1(T,x)=
      v_1(T,x)$ for almost every $x\leq 0$. 
      Then  $\int_0^T f(u_1(t,0))\,dt=\int_0^T f(v_1(t,0))\,dt$.

    \item Let  $u_2, v_2$
      be two entropy admissible  weak solutions to \eqref{bdry-outgoing}, with the same  initial datum $\overline u\in \mathbf{L^\infty}
      ([0,+\infty); \R)$ and 
      such that $u_2(T,x)=
      v_2(T,x)$ for almost every $x\geq 0$. 
      Then  $\int_0^T f(u_2(t,0))\,dt = \int_0^T f(v_2(t,0))\, dt$. 
  \end{enumerate} 
\end{lemma}

\begin{proof} 
We prove only statement $1$, since the proof of $2$ is completely analogous. 

Take $L > \max\left\{ \|f'(u_1)\|_{\LL\infty([0,T]\times (-\infty, 0))}, 
  \|f'(u_2)\|_{\LL\infty([0,T]\times (-\infty, 0))}\right\}$ and observe that for all $x<-LT$, the  values of  entropy  
   admissible  weak  solutions $u_1,v_1$ to~\eqref{bdry-incoming} are
determined only by the initial data $\overline u$: this means that $f(u_1(t,x))=f(u_2(t,x))$ for a.e. $t\in [0,T]$ and $x<-LT$.

Then, by the Divergence Theorem   we find that 
\begin{multline*} \int_0^T f(u_1(t,0))dt= \int_0^T f(u_1(t,-LT^+))dt+\int_{-LT}^0 (\overline u(x)-u_1(T,x)) dx
\\ = \int_0^T f(v_1(t,-LT^+))dt+\int_{-LT}^0 (\overline u(x)-v_1(T,x)) dx=\int_0^T
f(v_1(t,0))dt.\end{multline*} 
\end{proof} 

 Now we are ready to prove \Cref{maxtheo1}. We repeat the statement for more clarity.  

\begin{theorem}
  \label{teo:classical-sol} 
  Given $\overline u\in \mathbf{L^\infty}(\R; \R)$,
  let $\widetilde u$ be the unique entropy admissible solution
  to~\eqref{eq:CP-classic}. 
  Then, for every $T>0$, we have that  $f(\widetilde u(\cdot, 0))$
  solves the maximization problem~\eqref{max-J}. 
  More precisely, there holds   
  \begin{equation*}
    \int_0^T f(\widetilde u(t,0)) \, dt =
    \max_{\gamma \in {\mathcal{U}}} \int_0^T \gamma(t)\,  dt.
  \end{equation*}
\end{theorem}
 
\begin{proof}
  Consider an arbitrary $\gamma\in \mathcal{U}$, and define the corresponding solution 
  $u(t,x; \gamma)$ as in~\eqref{ugamma-2} in \Cref{defiammissibile}.  
  From now on, we denote $u_1(t,x) = u(t,x; \gamma)$ for $x<0$ and 
  $u_2(t,x) = u(t,x; \gamma)$ for $x>0$.
  Note that, by \Cref{defiammissibile}, 
  \[f(u_2(t,0)) = f(u_1(t,0))\qquad\text{for a.e. $t\in [0,T]$}.\]

  Let $T>0$ be fixed. 
  Consider the minimal backward characteristic $\widetilde \xi$ for $\widetilde u$ 
  passing at $(T,0)$. Without loss of generality we may assume  that $\widetilde \xi(0)\leq 0$,
  which implies that all the values of $\widetilde{u}(T, x)$, for $x < 0$,
  are determined by $\overline u_{|(-\infty, 0)}$. 
  We claim that
  \begin{equation}
    \label{claim}
    \int_0^{T} f\left( \widetilde u(s, 0)\right)\, ds \geq 
    \int_0^T f\left(u_1(s, 0)\right)\, ds.
  \end{equation}
  In the case $\widetilde \xi(0) > 0$, one
  can prove with the same arguments that
  \[
    \int_0^{T} f\left( \widetilde u(s, 0)\right)\, ds \geq 
    \int_0^T f\left(u_2(s, 0)\right)\, ds.
  \]

  Assume first that there exists a 
  backward characteristic $\xi$ for $u_1$ passing at $(T,0^-)$ with  $\xi(0)<0$. 
  In this case we have that the value of $u$ on the line $(T, x)$, 
  for $x\leq 0$, is completely determined by the initial datum $\overline u$: 
  therefore $\widetilde u(T,x) = u_1(T,x)$ for a.e. $x\leq 0$. By Lemma \ref{uniq} applied to $\widetilde u, u_1$ we get  
  \begin{align*}
    \int_0^{T} f(u_1(t,0))dt
    =\int_0^{T} f(\widetilde u(t,0))dt,
  \end{align*}
  which implies the claim~\eqref{claim}.  

  Assume now that all the 
  backward characteristics $\xi$ for $u_1$ passing at $(T,0^-)$ satisfy $\xi(0) \ge 0$. 
  Define now the point
  \begin{equation*}
    \bar x = \inf \left\{x\leq 0\colon\,  f'(u_1(T,z^+))T \le z \, 
    \text{ for a.e. $z\in (x,0)$}\right\}.
  \end{equation*}  
  We deduce that  necessarily $u_1(T,x)=\widetilde u(T,x)$ for a.e. $x< \bar x$. 
  Moreover,  $u_1(T,x)> \theta$ for a.e. $x\in (\bar x, 0)$ 
  since all characteristics for $u_1$ passing at $(T,x)$ for $x\in (\bar x, 0)$ must cross $x=0$ and have strictly negative slope.
  This implies that $u_1(T,x)\geq \widetilde u(T,x)$ for a.e. $x\in (\bar x, 0)$ since either $\widetilde u(T,x)\leq \theta$
   or the negative slope of the backward characteristics for $\widetilde u$ at $(T,x)$ is higher. 
  In conclusion choosing $L$ as above, 
  \[\int_{-LT}^0 \widetilde u(T,x)dx\leq \int_{-LT}^0  u_1(T,x)dx. \]
  So, repeating the argument based on the Divergence Theorem 
  we have that
  \begin{align*}
    \int_0^{T} f(u_1(t,0))dt
    & = \int_0^{T} f(u_1(t, -LT))dt + \int_{-LT}^0 \overline u(x) dx -\int_{-LT}^0 u_1(T,x)dx
    \\
    & \leq  \int_0^{T} f(\widetilde u(t, -LT))dt + \int_{-LT}^0  \overline u (x)dx 
    -\int_{-LT}^0 \widetilde u(T,x)dx
    = \int_0^{T} f(\widetilde u(t,0))dt
  \end{align*}
  proving the claim \eqref{claim}. 
\end{proof}

\section{Non entropic solutions to the  min-max problem (\ref{maxmin})}
\label{sectionesempi} 
In this section we propose two examples showing that, in general, the entropy solution
to~\eqref{eq:CP-classic} does not satisfy the min-max problem~\eqref{maxmin}.
They are characterized by an initial datum $\overline{u}$ 
piecewise constant but not monotone. In \Cref{shrar} we consider
an initial condition $\overline{u}$ piecewise constant with two points
of discontinuity located at the left of the interface $x=0$.
Instead, in \Cref{ex:Fabio} one point of discontinuity is located at the left of the
interface, while the second one at the right.
The general idea consists in producing waves from the interface $x=0$,
which interact with 
shocks and rarefaction fronts coming from the initial data, so to
achieve the maximum amount of possible cancellations.

\begin{example}
  \label{shrar} 
  \upshape 
  Consider the Cauchy problem~\eqref{eq:CP-classic} with initial condition
  \begin{align*}
    \overline{u} (x) & = \left\{
      \begin{array}{l@{\qquad}l}
        a_1, & \textrm{if}\,\, x <  \overline{x}_1,
        \\
        a_2, & \textrm{if}\,\,  \overline{x}_1 < x <  \overline{x}_2,
        \\
        a_3, & \textrm{if}\,\, x >  \overline{x}_2,
      \end{array}
    \right.
  \end{align*}
  where $\overline{x}_1 <  \overline{x}_2 < 0$ and
  \begin{equation}
    \label{eq:ex-3-ass1}
    a_1 < a_3 < a_2 < \theta;
  \end{equation}
  see Figure~\ref{fig:ex-3-flux}.
  \begin{figure} 
    \centering
    \begin{tikzpicture}[line cap=round,line join=round,x=1.cm,y=1.cm]
  
      \draw[->, line width=1.1pt] (2., .5) -- (2., 5.) node[below right]{$f$};
      \draw[->, line width=1.1pt] (2., .5) -- (11., .5) node[above]{$u$};

      \draw[variable=\x, domain=2:10, line width=1.1pt, name path=flow] plot ({\x}, {0.25*(\x -2)*(10 - \x) + 0.5});

      \node[inner sep=0, anchor=north] at (6, 0.4) {$\theta$};

      \draw[name path=theta, draw=none] (6, 0.5) -- (6, 5);

      \path[name intersections={of=flow and theta, by=P1}];

      \draw[color=blue, fill=blue] (P1) circle (2pt);
      \draw[dashed] (6, 0.5) -- (P1);

      \node[inner sep=0, anchor=north] at (3, 0.4) {$a_1$};
      \draw[name path=a1, draw=none] (3, 0.5) -- (3, 5);
      \path[name intersections={of=flow and a1, by=P2}];
      \draw[color=blue, fill=blue] (P2) circle (2pt);
      \draw[dashed] (3, 0.5) -- (P2);

      \newdimen\yfa
      \pgfextracty{\yfa}{\pgfpointanchor{P2}{center}}
      \draw[dashed] (2,\yfa) node[left]{$f(a_1)$}-- (10, \yfa);

      \node[inner sep=0, anchor=north] at (5, 0.4) {$a_2$};
      \draw[name path=a2, draw=none] (5, 0.5) -- (5, 5);
      \path[name intersections={of=flow and a2, by=P3}];
      \draw[color=blue, fill=blue] (P3) circle (2pt);
      \draw[dashed] (5, 0.5) -- (P3);

      \newdimen\yfa
      \pgfextracty{\yfa}{\pgfpointanchor{P3}{center}}
      \draw[dashed] (2,\yfa) node[left]{$f(a_2)$}-- (10, \yfa);

      \node[inner sep=0, anchor=north west] at (4, 0.4) {$a_3$};
      \draw[name path=a3, draw=none] (4, 0.5) -- (4, 5);
      \path[name intersections={of=flow and a3, by=P4}];
      \draw[color=blue, fill=blue] (P4) circle (2pt);
      \draw[dashed] (4, 0.5) -- (P4);

      \newdimen\yfa
      \pgfextracty{\yfa}{\pgfpointanchor{P4}{center}}
      \draw[dashed] (2,\yfa) node[left]{$f(a_3)$}-- (10, \yfa);

      \node[inner sep=0, anchor=north] at (8.3, 0.4) {$b$};
      \draw[name path=b, draw=none] (8.3, 0.5) -- (8.3, 5);
      \path[name intersections={of=flow and b, by=P5}];
      \draw[color=blue, fill=blue] (P5) circle (2pt);
      \draw[dashed] (8.3, 0.5) -- (P5);

      \newdimen\yfa
      \pgfextracty{\yfa}{\pgfpointanchor{P5}{center}}
      \draw[dashed, name path=fc] (2,\yfa) node[below left]{$f(b) = f(c)$}-- (10, \yfa);
      \path[name intersections={of=flow and fc, by=P6}];
      \draw[color=blue, fill=blue] (P6) circle (2pt);
      \newdimen\yfb
      \pgfextractx{\yfb}{\pgfpointanchor{P6}{center}}
      \draw[dashed] (\yfb, 0.5) -- (\yfb, \yfa);
      \node[inner sep=0, anchor=north east] at (\yfb, 0.4) {$c$};
    \end{tikzpicture}
    \caption{The configurations of \Cref{shrar}.}
    \label{fig:ex-3-flux}
  \end{figure}
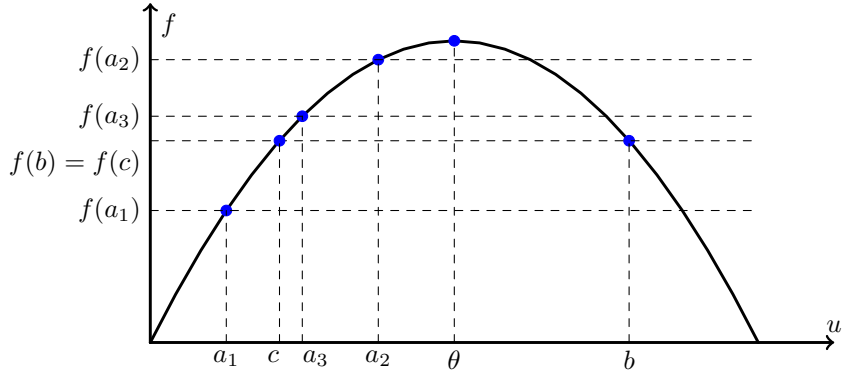
  Suppose also that 
  \begin{equation}
    \label{eq:ex-3-ass2}
    - \frac{ \overline{x}_2}{f'(a_2)} < -\frac{ \overline{x}_1}{\lambda} < T,
  \end{equation}
  where $\lambda = \frac{f(a_2) - f(a_1)}{a_2 - a_1}$.
  The discontinuity at the point $\overline x_1$ originates a shock wave with positive
  speed given by $\lambda$. Instead at the point $\overline x_2$ a rarefaction wave with positive speed originates.
  Assumption~\eqref{eq:ex-3-ass2} implies that the shock wave exiting $\overline{x}_1$
  reaches the interface $x=0$ before interacting with the rarefaction wave.
  Hence such interaction happens for $x > 0$ generating waves with positive speed
  due to~\eqref{eq:ex-3-ass1}. Therefore
  the entropy solution $\widetilde u$ in the region $x < 0$
  is given by
  \begin{equation}
    \label{eq:entropy_solution_example1}
    \widetilde u(t,x) =
    \left\{
      \begin{array}{l@{\quad}l}
        a_1, & \textrm{ if } x <  \overline{x}_1 + \lambda t,
        \\
        a_2, & \textrm{ if }  \overline{x}_1 + \lambda t < x <  \overline{x}_2 + f'(a_2) t,
        \\
        (f')^{-1} \left(\frac{x -  \overline{x}_2}{t}\right),  
        & \textrm{ if } f'(a_2) t < x -  \overline{x}_2 < f'(a_3) t,
        \\
        a_3, & \textrm{ if } x >  \overline{x}_2 + f'(a_3) t;
      \end{array}
    \right.
  \end{equation}
  see \Cref{fig:ex-3-fig1}.
  Denote with $\widetilde{\gamma}(t) = f \left( \widetilde{u}(t, 0^-)\right)$ the boundary
  control associated with the entropy solution $\widetilde{u}$. Note that $\widetilde{\gamma}(t)$ is not monotone.   Recalling the definition of
  the functional $\mathcal F_{[0, T]}$, see~\eqref{tv-def}, we get
  \begin{equation}
    \label{eq:F-entropy-ex1}
    \mathcal{F}_{[0,T]} (\widetilde \gamma) = 2f(a_2)-f(a_3)-f(a_1).
  \end{equation}
  
  Choose boundary data $b$ and $c$ such 
  \begin{equation}\label{prova1}
    c < \theta < b, \qquad f(a_1) < f(b) = f(c) \leq  f(a_3)<f(a_2),\qquad f(a_3)-f(b)<f(a_2)-f(a_3),
  \end{equation}
  see \Cref{fig:ex-3-flux}, and consider the constant boundary data $k_1= b, k_2=c$. 

  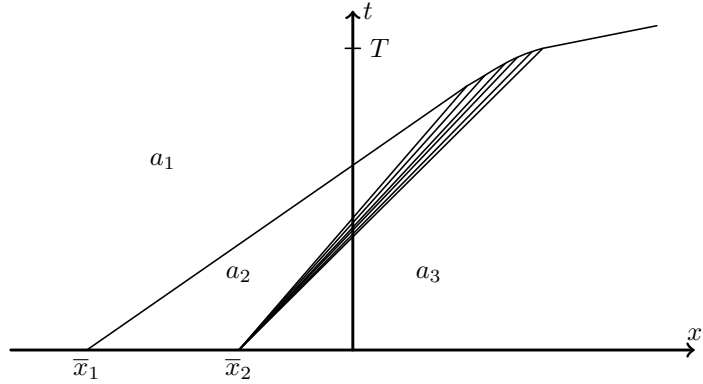
\begin{figure}
    \centering
    \begin{tikzpicture}[line cap=round,line join=round,x=1.cm,y=1.cm]
  
      \draw[->, line width=1.1pt] (6.5, .5) -- (6.5, 5.) node[right]{$t$};
      \draw[->, line width=1.1pt] (2., .5) -- (11., .5) node[above]{$x$};

      \node[inner sep=0, anchor=north] at (3, 0.4) {$\overline{x}_1$};
      \node[inner sep=0, anchor=north] at (5, 0.4) {$\overline{x}_2$};



      \draw[line width=.6pt] (3, 0.5) -- (8, 4) to [out=30, in=195] (9., 4.5) -- (10.5, 4.8);
      \draw[line width=.6pt] (5, 0.5) -- (8, 4) ;
      \draw[line width=.6pt] (5, 0.5) -- (9, 4.5);
      \begin{scope}
        \clip (3, 0.5) -- (8, 4) to [out=30, in=195] (9., 4.5) -- (10.5, 4.8) -- (10.5, 0.5) -- cycle;
        \draw[line width=.6pt] (5, 0.5) -- (9., 5) ;
        \draw[line width=.6pt] (5, 0.5) -- (9.15, 5) ;
        \draw[line width=.6pt] (5, 0.5) -- (9.25, 5) ;
        \draw[line width=.6pt] (5, 0.5) -- (9.39, 5) ;
      \end{scope}

      \node[inner sep=0] at (4, 3) {$a_1$};
      \node[inner sep=0] at (5, 1.5) {$a_2$};
      \node[inner sep=0] at (7.5, 1.5) {$a_3$};

      \draw[line width=.6pt] (6.4, 4.5) -- (6.6, 4.5) node[right]{$T$};

    \end{tikzpicture}
    \caption{The entropy solution $\widetilde u$, defined in~\eqref{eq:entropy_solution_example1},
      of Example~\ref{shrar}.}
    \label{fig:ex-3-fig1}
  \end{figure}
  The choices~\eqref{eq:ex-3-ass1} and~\eqref{prova1} imply that
  the wave generated by the jump $(a_1, b)$
  is a shock with positive speed, while
  the waves generated by the jumps $(a_2, b)$ and $(a_3, b)$
  are both shocks with negative speed. 
  Note that it is possible to select the points $\overline{x}_1$ and $\overline{x}_2$ such that
  the wave generated at $\overline{x}_1$ arrives at the interface $x=0$,
  after interacting with the wave $(a_2, b)$, at time $\tau < T$.
  Therefore the maps $u_1$ and $u_2$ solving respectively the initial-boundary value problems~\eqref{bdry-incoming}
  and~\eqref{bdry-outgoing} with controls $k_i$ are
  \begin{equation}
    \label{eq:u_1_example1}
    u_1(t,x) = 
      \begin{cases}
        a_1, &\qquad \textrm{ if } (t,x) \in A_1,
        \\
        a_2, &\qquad \textrm{ if } (t,x) \in A_2, 
        \\
        (f')^{-1} \left(\frac{x -  \overline{x}_2}{t}\right),  
        &\qquad \textrm{ if } (t,x) \in A_4,
        \\
        a_3, &\qquad \textrm{ if } (t,x) \in A_3, 
        \\
        b, &\qquad \textrm{ otherwise },
      \end{cases}
      \end{equation}
  and
  \begin{equation}
    \label{eq:u_2_example1}
    u_2(t,x)  =
      \begin{cases}          
        c, &\qquad \textrm{ if } \max\left\{0, \frac{f(c) - f(a_1)}{c - a_1}
          \, (t - \tau)\right\} < x < \frac{f(c) - f(a_3)}{c - a_3} \, t,
        \vspace{.2cm}\\
        a_3, &\qquad \textrm{ if } x > \frac{f(c) - f(a_3)}{c - a_3} \, t,
        \vspace{.2cm}\\
        a_1, &\qquad \textrm{ if } x < \frac{f(c) - f(a_1)}{c - a_1} \, (t - \tau),
      \end{cases}
  \end{equation}
  where 
  the sets $A_1$, $A_2$, $A_3$, $A_4$ are suitable subsets of the region $x<0$,
  see \Cref{fig:ex-3-fig-2}. 
  
  Note that $\gamma(t)=f(u_1(t, 0))=f(u_2(t, 0))\in \mathcal{U}$. Moreover, since $u_1(T,x)=\widetilde u(T,x)$ for a.e. $x\leq 0$,  by Lemma \ref{uniq} and Theorem \ref{teo:classical-sol}, $\gamma$  is a solution 
  to the maximization problem~(\ref{max-J}). Moreover, by~\eqref{prova1},
  \begin{equation}
    \label{eq:F-nonentropy-ex1}
    \mathcal{F}_{[0,T]} (\gamma) = \abs{f(b) - f(a_1)} + \abs{f(a_3) - f(b)} + \abs{f(a_3) - f(c)}
    =  2f(a_3)-f(a_1) -f(b).
  \end{equation}
  Hence~\eqref{eq:F-entropy-ex1},~\eqref{prova1}, and~\eqref{eq:F-nonentropy-ex1} imply that
  \begin{align*}
    \mathcal{F}_{[0,T]} (\gamma)
    -\mathcal{F}_{[0,T]} (\widetilde \gamma) =
    &   2f(a_3)-f(a_1)-f(b)-2f(a_2)+f(a_3)+f(a_1)
    \\
    = & 2f(a_3)-2f(a_2)+f(a_3)-f(b)
    \\
    < & 2f(a_3)-2f(a_2) + f(a_2) - f(a_3)=f(a_3) - f(a_2) <0, 
  \end{align*}
  showing that the entropy solution $\widetilde{u}$ is not optimal for the functional
  $\mathcal F_{[0,T]}$.

We conclude observing that  in this case we may compute explicitly  the solution of the minimization problem \eqref{maxmin}. Indeed if we take $c=a_3$ in \eqref{prova1}, we obtain \[\mathcal{F}_{[0,T]} (\gamma)=f(a_3)-f(a_1)=\min_{\gamma\in\mathcal{U}_{max}^M} \mathcal{F}_{[0,T]} (\gamma).\]  
The shock discontinuity generated by $(a_3,b)$, with $a_3=c$, emerges tangentially from
the junction $x = 0$.
 
  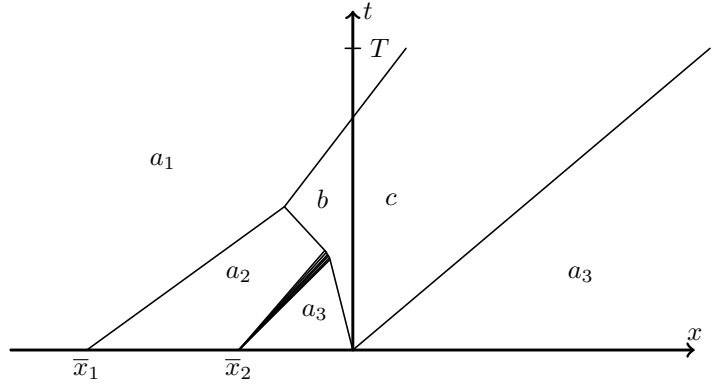
\begin{figure}
    \centering
    \begin{tikzpicture}[line cap=round,line join=round,x=1.cm,y=1.cm]
  
      \draw[->, line width=1.1pt] (6.5, .5) -- (6.5, 5.) node[right]{$t$};
      \draw[->, line width=1.1pt] (2., .5) -- (11., .5) node[above]{$x$};

      \node[inner sep=0, anchor=north] at (3, 0.4) {$\overline{x}_1$};
      \node[inner sep=0, anchor=north] at (5, 0.4) {$\overline{x}_2$};



      \draw[line width=.6pt] (6.5, 0.5) -- (6.2, 1.7) to [out=100, in=290] (6.15, 1.8)
      -- (5.6, 2.4);

      \draw[line width=.6pt] (3, 0.5) -- (5.6, 2.4);
      \begin{scope}
        \clip (6.5, 0.5) -- (6.2, 1.7) to [out=100, in=290] (6.15, 1.8)
        -- (5.6, 2.4) -- (3, 0.5) -- cycle;
        \draw[line width=.6pt] (5, 0.5) -- (8, 4) ;
        \draw[line width=.6pt] (5, 0.5) -- (9, 4.5);
        \draw[line width=.6pt] (5, 0.5) -- (9., 5) ;
        \draw[line width=.6pt] (5, 0.5) -- (9.15, 5) ;
        \draw[line width=.6pt] (5, 0.5) -- (9.25, 5) ;
        \draw[line width=.6pt] (5, 0.5) -- (9.39, 5) ;
      \end{scope}

      \draw[line width=.6pt] (5.6, 2.4) -- (7.2, 4.5); 
      \draw[line width=.6pt] (6.5, 0.5) -- (11.2, 4.5); 
      \node[inner sep=0] at (4, 3) {$a_1$};
      \node[inner sep=0] at (5, 1.5) {$a_2$};
      \node[inner sep=0] at (9.5, 1.5) {$a_3$};
      \node[inner sep=0] at (6., 1.) {$a_3$};
      \node[inner sep=0] at (6.1, 2.5) {$b$};
      \node[inner sep=0] at (7., 2.5) {$c$};

      \draw[line width=.6pt] (6.4, 4.5) -- (6.6, 4.5) node[right]{$T$};

    \end{tikzpicture}
    \caption{The solution $u_1$ and $u_2$, defined in~\eqref{eq:u_1_example1} and
    in~\eqref{eq:u_2_example1}, of \Cref{shrar}.}
    \label{fig:ex-3-fig-2}
  \end{figure}
\end{example}

\begin{example}
  \label{ex:Fabio}
  \upshape 
  Consider the Cauchy problem~\eqref{eq:CP-classic} with initial condition
  \begin{align}\label{datoinizialees}
    \overline{u} (x) & = \left\{
      \begin{array}{l@{\qquad}l}
        a_1, & \textrm{if}\,\, x <  \overline{x}_1,
        \\
        a_2, & \textrm{if}\,\,  \overline{x}_1 < x <  \overline{x}_2,
        \\
        a_3, & \textrm{if}\,\, x >  \overline{x}_2,
      \end{array}
    \right.
  \end{align}
  where $\overline{x}_1 < 0 < \overline{x}_2$ and the densities
  \begin{equation}
    \label{eq:ex-fabio-ass1}
    a_1 <  \theta < a_3 <  a_2,\quad\text{with} \quad 
    f(a_1) < f(a_2)  
    < f(a_3) ;
  \end{equation}
  see \Cref{fig:ex-fabio-flux}.
  \begin{figure} 
    \centering
    \begin{tikzpicture}[line cap=round,line join=round,x=1.cm,y=1.cm]
  
      \draw[->, line width=1.1pt] (2., .5) -- (2., 5.) node[below right]{$f$};
      \draw[->, line width=1.1pt] (2., .5) -- (11., .5) node[above]{$u$};

      \draw[variable=\x, domain=2:10, line width=1.1pt, name path=flow] plot ({\x}, {0.25*(\x -2)*(10 - \x) + 0.5});

      \node[inner sep=0, anchor=north] at (6, 0.4) {$\theta$};

      \draw[name path=theta, draw=none] (6, 0.5) -- (6, 5);

      \path[name intersections={of=flow and theta, by=P1}];

      \draw[color=blue, fill=blue] (P1) circle (2pt);
      \draw[dashed] (6, 0.5) -- (P1);

      \node[inner sep=0, anchor=north] at (3, 0.4) {$a_1$};
      \draw[name path=a1, draw=none] (3, 0.5) -- (3, 5);
      \path[name intersections={of=flow and a1, by=P2}];
      \draw[color=blue, fill=blue] (P2) circle (2pt);
      \draw[dashed] (3, 0.5) -- (P2);

      \newdimen\yfa
      \pgfextracty{\yfa}{\pgfpointanchor{P2}{center}}
      \draw[dashed] (2,\yfa) node[left, name path=fa1]{$f(a_1)$}-- (10, \yfa);
      \draw[draw=none, name path=fa1] (7,\yfa) -- (10, \yfa);

      \path[name intersections={of=flow and fa1, by=FA1}];
      \draw[color=blue, fill=blue] (FA1) circle (2pt);
      \newdimen\yfb
      \pgfextractx{\yfb}{\pgfpointanchor{FA1}{center}}
      \draw[dashed] (\yfb, 0.5) node[below]{$\overline{a}_1$} -- (FA1);

      \node[inner sep=0, anchor=north] at (8.5, 0.4) {$a_2$};
      \draw[name path=a2, draw=none] (8.5, 0.5) -- (8.5, 5);
      \path[name intersections={of=flow and a2, by=P3}];
      \draw[color=blue, fill=blue] (P3) circle (2pt);
      \draw[dashed] (8.5, 0.5) -- (P3);

      \newdimen\yfa
      \pgfextracty{\yfa}{\pgfpointanchor{P3}{center}}
      \draw[dashed, name path=fa2] (2,\yfa) node[left]{$f(a_2)$}-- (10, \yfa);

      \path[name intersections={of=flow and fa2, by=FA2}];
      \draw[color=blue, fill=blue] (FA2) circle (2pt);
      \newdimen\yfb
      \pgfextractx{\yfb}{\pgfpointanchor{FA2}{center}}
      \draw[dashed] (\yfb, 0.5) node[below]{$\overline{a}_2$} -- (FA2);

      \node[inner sep=0, anchor=north] at (7.5, 0.4) {$a_3$};
      \draw[name path=a3, draw=none] (7.5, 0.5) -- (7.5, 5);
      \path[name intersections={of=flow and a3, by=P4}];
      \draw[color=blue, fill=blue] (P4) circle (2pt);
      \draw[dashed] (7.5, 0.5) -- (P4);

      \newdimen\yfa
      \pgfextracty{\yfa}{\pgfpointanchor{P4}{center}}
      \draw[dashed, name path=fa3] (2,\yfa) node[left]{$f(a_3)$}-- (10, \yfa);

      \path[name intersections={of=flow and fa3, by=FA3}];
      \draw[color=blue, fill=blue] (FA3) circle (2pt);
      \newdimen\yfb
      \pgfextractx{\yfb}{\pgfpointanchor{FA3}{center}}
      \draw[dashed] (\yfb, 0.5) node[below]{$\overline{a}_3$} -- (FA3);

      \node[inner sep=0, anchor=north] at (8., 0.4) {$\overline a_4$};
      \draw[name path=a4, draw=none] (8, 0.5) -- (8, 5);
      \path[name intersections={of=flow and a4, by=P5}];
      \draw[color=blue, fill=blue] (P5) circle (2pt);
      \draw[dashed] (8, 0.5) -- (P5);

      \newdimen\yfa
      \pgfextracty{\yfa}{\pgfpointanchor{P5}{center}}
      \draw[dashed, name path=fa4] (2,\yfa) node[left]{$f(a_4)$}-- (10, \yfa);

      \path[name intersections={of=flow and fa4, by=FA4}];
      \draw[color=blue, fill=blue] (FA4) circle (2pt);
      \newdimen\yfb
      \pgfextractx{\yfb}{\pgfpointanchor{FA4}{center}}
      \draw[dashed] (\yfb, 0.5) node[below]{${a}_4$} -- (FA4);

    \end{tikzpicture}
    \caption{The configuration of \Cref{ex:Fabio}.}
    \label{fig:ex-fabio-flux}
  \end{figure}
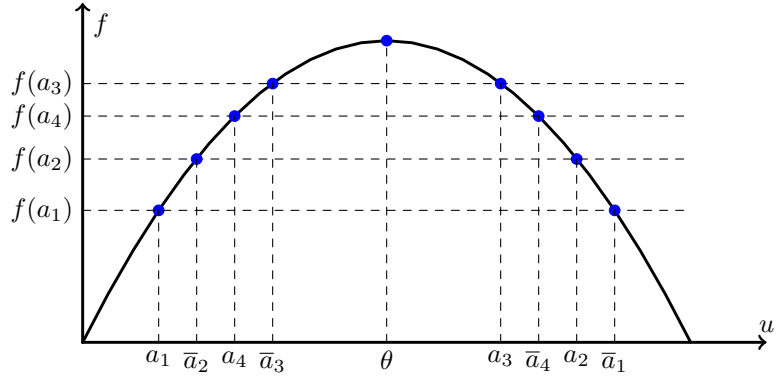

  The discontinuity at the point $\overline x_1$ originates a shock wave with positive
  speed given by $\lambda = \frac{f(a_1) - f(a_2)}{a_1 - a_2}$, 
  whereas at the point $\overline x_2$ a rarefaction wave with negative speed originates.
  Choosing appropriately the points $\overline x_1$ and $\overline x_2$, by~\eqref{eq:ex-fabio-ass1}, 
  we may assume that  the rarefaction wave interacts
  with the shock wave in the semiplane $x<0$, being 
  completely canceled in the semiplane $x<0$,  and the resulting shock wave interacts with the interface $x=0$
  at time $t_3>t_2$; see \Cref{fig:ex-fabio-fig1}.  We define the first and the latter time of interaction of the rarefaction wave
  exiting $\overline{x}_2$ with $x=0$ as $t_1,t_2$: 
  \begin{equation}
    \label{eq:times_ex_2}
        t_1  \doteq  - \frac{\overline{x}_2}{f'(a_2)}<   - \frac{\overline{x}_2}{f'(a_3)}\doteq t_2.   
  \end{equation}
  We may compute explicitly the value $t_3$ as follows: 
  \begin{equation} 
    \label{t3}
    t_3  \doteq   \frac{a_1 \overline{x}_1 - a_3 
    \overline{x}_2 + a_2(\overline{x}_2 - \overline{x}_1)}{f(a_3) - f(a_1)}.
  \end{equation}
  The   explicit formula~\eqref{t3} can be obtained by  applying the Divergence Theorem to the vector field $(\widetilde u, f(\widetilde u))$ in  the  rectangular  domain 
 $(0, t_3) \times (\overline{x}_1, \overline{x}_2)$. 

  Recall that $t_3>t_2$. Without loss of generality we can assume that $t_3< T$; 
  see \Cref{fig:ex-fabio-fig1}. 
  \begin{figure}
    \centering
    \begin{tikzpicture}[line cap=round,line join=round,x=1.cm,y=1.cm]
  
      \draw[->, line width=1.1pt] (6.5, .5) -- (6.5, 5.) node[right]{$t$};
      \draw[->, line width=1.1pt] (2., .5) -- (11., .5) node[above]{$x$};

      \node[inner sep=0, anchor=north] at (4., 0.4) {$\overline{x}_1$};
      \node[inner sep=0, anchor=north] at (9, 0.4) {$\overline{x}_2$};

      \draw[line width=.6pt] (4., 0.5) -- (4.5, 2) to [out=70, in=195] (5.5, 3.) -- (10., 4.5);
      \draw[line width=.6pt] (9, 0.5) -- (4.5, 2) ;
      \draw[line width=.6pt] (9, 0.5) -- (5.5, 3.);
      \begin{scope}
        \clip (4., 0.5) -- (4.5, 2) to [out=70, in=195] (5.5, 3.) -- (10., 4.5) -- (10.5, 0.5) -- cycle;
        \draw[line width=.6pt] (9, 0.5) -- (4., 2.4) ;
        \draw[line width=.6pt] (9, 0.5) -- (4., 2.6) ;
        \draw[line width=.6pt] (9, 0.5) -- (4., 2.8) ;
        \draw[line width=.6pt] (9, 0.5) -- (4., 3.) ;
        \draw[line width=.6pt] (9, 0.5) -- (4., 3.2) ;
        \draw[line width=.6pt] (9, 0.5) -- (4., 3.4) ;
        \draw[line width=.6pt] (9, 0.5) -- (4., 3.6) ;
        \draw[line width=.6pt] (9, 0.5) -- (4., 3.8) ;
      \end{scope}

      \node[inner sep=0] at (4, 3) {$a_1$};
      \node[inner sep=0] at (5, 1.) {$a_2$};
      \node[inner sep=0] at (8.5, 2.5) {$a_3$};
      \node[inner sep=0, anchor=east] at (6.4, 1.1) {$t_1$};
      \node[inner sep=0, anchor=west] at (6.6, 2.4) {$t_2$};
      \node[inner sep=0, anchor=east] at (6.4, 3.5) {$t_3$};

      \draw[line width=.6pt] (6.4, 4.5) -- (6.6, 4.5) node[right]{$T$};

    \end{tikzpicture}
    \caption{The entropy solution $\widetilde u$ of \Cref{ex:Fabio}.
    At the location $\overline{x}_1$ a shock wave, connecting the states $a_1$ and $a_2$, originates.
    At the position $\overline{x}_2$ a rarefaction wave, connecting $a_2$ and $a_3$, is generated.
    This wave interacts with the interface $x=0$ in the time interval $[t_1, t_2]$. In the
    semiplane $x<0$ the shock wave interacts with the rarefaction one and the resulting
    shock travels with positive speed and reaches the interface $x=0$ at time $t_3 < T$.}
    \label{fig:ex-fabio-fig1}.
  \end{figure}
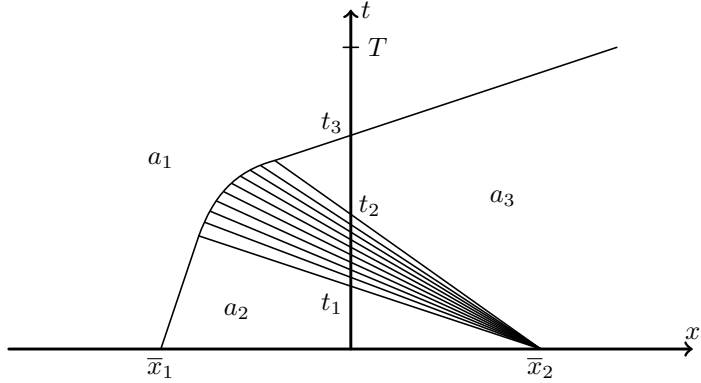

  Denote with $\widetilde{\gamma}(t) = f \left( \widetilde{u}(t, 0)\right)$ the boundary
  control associated with the entropy solution $\widetilde{u}$. Note that $\widetilde{\gamma}(t)$ is nondecreasing in $(0, t_2)$  and nonincreasing in $(t_2, T)$ and moreover $\max \widetilde \gamma=f(a_3).$ Recalling the definition of
  the functional $\mathcal F_{[0, T]}$, see~\eqref{tv-def}, we get
  \begin{equation}
    \label{eq:F-entropy-exFabio}
    \mathcal{F}_{[0,T]} (\widetilde \gamma) = 2f(a_3)-f(a_1)-f(a_2).
  \end{equation}
Note that in every interval $(0, t)$ with $t_3<t\leq T$ the flow $f(\widetilde u(t,0))$ is not monotone.

  We provide now a construction for a boundary inflow control  $\gamma(t)$ which is a solution
  to the maximization problem~(\ref{max-J}) and 
  $\mathcal{F}_{[0,T]} (\gamma) < \mathcal{F}_{[0,T]} (\widetilde \gamma)$. 

  The construction  of the boundary data is quite involved and interesting on its own in our opinion, 
  so  we   provide it in the following Proposition. 
  \begin{proposition}
    \label{propesempiofabio}
    Let $\overline u$ as defined in~\eqref{datoinizialees}. 
    Then there exist boundary data $k_1, k_2$ such that the solutions $u_1, u_2$ 
    to the boundary value problems~\eqref{bdry-incoming}, \eqref{bdry-outgoing} 
    are associated to an admissible boundary inflow control $\gamma(t)= f(u_1(t, 0))=f(u_2(t,0))$  which 
    is a solution 
    to the maximization problem~\eqref{max-J} and satisfies 
    $\mathcal{F}_{[0,T]} (\gamma) < \mathcal{F}_{[0,T]} (\widetilde \gamma)$.

    In particular, the entropy admissible solution $\widetilde u$ with initial data 
    $\overline u$ is not a solution to min-max problem~\eqref{maxmin}.
  \end{proposition}

  \begin{proof} Define $\overline{a}_2, \overline a_3 $ such that $\overline a_2 < \overline a_3<\theta$ 
    and $f(a_2) = f(\overline{a}_2),\ f(a_3) = f(\overline{a}_3)$; see~\Cref{fig:ex-fabio-flux}.

    We aim to construct  the boundary data $k_i$, depending on   a new time  $t_4\in (t_1, t_2)$ 
    and on densities $a_4$ and $\overline{a}_4$  with
    \begin{equation}
      \label{eq:ex-fabio-ass2}
      \begin{split}
        a_1< \overline a_2 & <a_4<\overline a_3<\theta<a_3<\overline a_4<a_2 < \overline a_1,  \\  
        f(a_1) = f(\overline a_1) <f(a_2)=&f(\overline a_2)<f(a_4)=f(\overline a_4)<f(a_3)=f(\overline a_3),
      \end{split}
    \end{equation}   
    see \Cref{fig:ex-fabio-flux},  such that the flux at $x=0$ of the corresponding solutions $u_1,u_2$ is given by
    \begin{equation*}
      \label{eq:boundary-ex-fabio}
      f(u_1(t,0)) = f(u_2(t,0))=\left\{
        \begin{array}{l@{\qquad}l}
          f(a_2),
          & t \in (0, t_4),
          \\
          f(a_4),
          & t \in (t_4, t_6),
          \\
        f(a_1),
          & t \in (t_6, T),
        \end{array}
      \right.
    \end{equation*} 
    where $t_6$ is defined as 
    \begin{equation}
      \label{t_6}
      t_6  \doteq \displaystyle \frac{-(a_2 - a_1)\overline{x}_1 + t_4 (f(a_4) - f(a_2))}{f(a_4) - f(a_1)}.
    \end{equation} 
    
    More precisely the solution $u_1$ would enjoy the following properties
    \begin{enumerate}
    \item[1.1] Up to the time $t_1$, its flux trace is constantly $f(u_1(t,0))=f(\widetilde u(t,0))$. 
      In the interval $(t_1, t_4)$, it holds $f(u_1(t,0)) = f(a_2)$.
    \item[1.2] At time $t_4$  a rarefaction wave, connecting the state $a_2$ with $\overline a_4$ 
      (recall~\eqref{eq:ex-fabio-ass2}), enters $x<0$.
    \item[1.3] In the interval $(t_4, t_6)$ the flow $f(u_1(\cdot,0))$ is constantly equal to $f(\overline a_4)$,  
      where $t_6$ is the time at which the shock wave obtained by the interaction between the shock originated 
      at $\overline x_1$ and the rarefaction originated at $t_4$ crosses the axis $x=0$, 
      see \Cref{fig:ex-fabio-fig2}. The time $t_6$ can be explicitely computed with the formula~\eqref{t_6},
      by applying the Divergence Theorem to the vector field $(u_1, f(u_1))$ 
      in the triangular domain with vertexes $(t_6,0)$, $(0, -f'(a_1)t_6)$ and $(0,0)$.

      We claim that actually $t_3<t_6<T  $, where $t_3$ is defined in~\eqref{t3} and $t_6$ is defined in~\eqref{t_6}.

    \item[1.4] In the interval $(t_6, T) $ the flow $f(u_1(t,0))$ is constantly 
      equal to $f(\widetilde u(t,0))=f(a_1)$.
    \end{enumerate}
    On the other side,  the features  of the solution $u_2$  would be the following:
    \begin{enumerate}

    \item[2.1] Up to time $t_1$ its flux trace is constantly $f(u_2(t,0))= f(\widetilde u(t,0)) = f(a_2)$.

    \item[2.2] At time $t_1$ a shock $\xi$, connecting the states $\overline a_2, a_2$ and entering $x>0$, is originated. 

      We claim that this shock remains for all times in $x>0$, never crossing  the axis $x=0$. 
    \item[2.3]   In the interval  $(t_1,t_4)$, it is constantly $f(u_2(t,0))= f(\overline a_2)=f(a_2)$.
    
    \item[2.4] In the interval $(t_4, t_6) $ the flow $f(u_2(t,0))$ is constantly equal to $f(a_4)$,  
      where $t_6 $ is defined in~\eqref{t_6}. 

    \item[2.5] In the interval $(t_6, T) $ the flow $f(u_2(t,0))$ is constantly equal to $f(\widetilde u(t,0))=f(a_1)$.
    \end{enumerate}
    
  Such   solutions $u_1,u_2$ are  represented in \Cref{fig:ex-fabio-fig2}.
  
  Observe that for these solutions we have that 
  $\gamma(t)=f(u_1(t,0))=f(u_2(t,0))$, and 
  \[
    \max_{t \in [0,T]}\gamma(t)=f(a_4)<f(a_3)=\max_{t \in [0,T]} \widetilde \gamma(t).
  \] 
  Moreover $u_1(t,0),u_2(t,0)$  coincide with $\widetilde u(t,0)$ on  $(0, t_1)\cup (t_6, T)$ 
  and $u_1(T,x)=\widetilde u(T,x)$ for a.e. $x<0$.  Therefore, due to Lemma \ref{uniq}, $\gamma(t)$ is a solution of the maximation 
  problem~\eqref{max-J}.
Moreover  $\mathcal{F}_{[0,T]} (\gamma)< \mathcal{F}_{[0,T]} (\widetilde \gamma)$. More explicitly  
  \[
    \mathcal{F}_{[0,T]} (\gamma) = \abs{f(a_2) - f(a_4)} + \abs{f(a_4) - f(a_1)}
    =  2f(a_4)-f(a_1) -f(a_2),
  \] 
  and hence recalling~\eqref{eq:F-entropy-exFabio}
  \[
    \mathcal{F}_{[0,T]} (\gamma)
    -\mathcal{F}_{[0,T]} (\widetilde \gamma) = 2f(a_4)-2f(a_3) <0.
  \]

  Therefore to conclude the proof it is sufficient to show that we can actually 
  choose $t_4, a_4$, and $\overline a_4$ so that  properties 1.1-1.4, 2.1-2.5 above are verified and in particular
  \begin{eqnarray}
    \label{claimfabio1}&t_3< t_6<T &
    \\ \label{claimfabio2}&
    \text{ the shock $\xi$ originated at $t_1$ and entering $x>0$  remains for all times in $x>0$. } &
  \end{eqnarray}    

    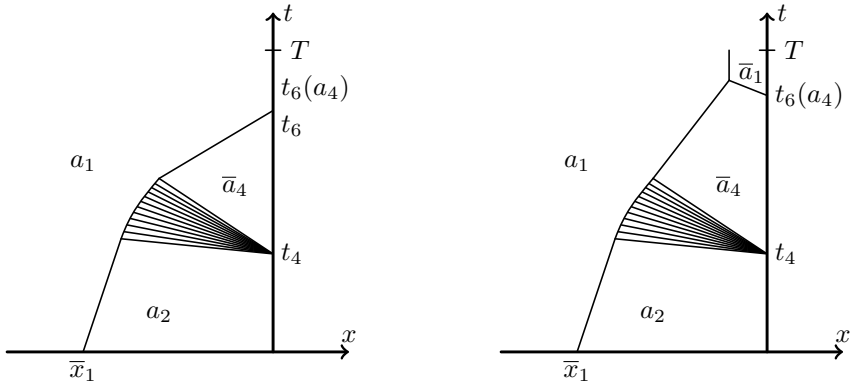
\begin{figure}
      \centering
      \begin{tikzpicture}[line cap=round,line join=round,x=1.cm,y=1.cm]
    
        \newcommand*{\incoming}{
          \draw[->, line width=1.1pt, name path=vertical] (6.5, .5) -- (6.5, 5.) node[right]{$t$};
          \draw[->, line width=1.1pt] (3., .5) -- (7.5, .5) node[above]{$x$};

          \node[inner sep=0, anchor=north] at (4., 0.4) {$\overline{x}_1$};
          
          \draw[line width=.6pt] (4., 0.5) -- (4.5, 2) to [out=70, in=230] (5., 2.8); 
          \draw[line width=.6pt, name path=rarefaction, draw=none] (9, 0.5) -- (4.5, 2) ;

          \path[name intersections={of=vertical and rarefaction, by=t1}];

          \draw[line width=.6pt] (6.5, 1.8) -- (4.5, 2);
          \draw[line width=.6pt] (6.5, 1.8) -- (5., 2.8);

          \node[inner sep=0] at (4, 3) {$a_1$};
          \node[inner sep=0] at (5, 1.) {$a_2$};
          \node[inner sep=0] at (6, 2.7) {$\overline a_4$};

          \begin{scope}
            \clip (6.5, 1.8) -- (4.5, 2) to [out=70, in=230] (5., 2.8) -- (6.5, 1.8);
            \draw[line width=.6pt] (6.5, 1.8) -- (4.5, 2.1);
            \draw[line width=.6pt] (6.5, 1.8) -- (4.5, 2.2);        
            \draw[line width=.6pt] (6.5, 1.8) -- (4.5, 2.3);
            \draw[line width=.6pt] (6.5, 1.8) -- (4.5, 2.4);
            \draw[line width=.6pt] (6.5, 1.8) -- (4.5, 2.5);
            \draw[line width=.6pt] (6.5, 1.8) -- (4.5, 2.6);
            \draw[line width=.6pt] (6.5, 1.8) -- (4.5, 2.7);
            \draw[line width=.6pt] (6.5, 1.8) -- (4.5, 2.8);
            \draw[line width=.6pt] (6.5, 1.8) -- (4.5, 2.9);
            \draw[line width=.6pt] (6.5, 1.8) -- (4.5, 3.);
          \end{scope}

          \node[inner sep=0, anchor=west] at (6.6, 1.8) {$t_4$};

          \draw[line width=.6pt] (6.4, 4.5) -- (6.6, 4.5) node[right]{$T$};
        }

        \begin{scope}[xshift=-2cm]
          \incoming
          \draw[line width=.6pt] (5., 2.8) -- (6.5, 3.7);
          \node[inner sep=0, anchor=west] at (6.6, 4.) {$t_6(a_4)$};
          \node[inner sep=0, anchor=west] at (6.6, 3.5) {$t_6$};

        \end{scope}

        \begin{scope}[xshift=4.5cm]
          \incoming
          \draw[line width=.6pt] (5., 2.8) -- (6., 4.1);
          \draw[line width=.6pt] (6.5, 3.9) -- (6., 4.1) -- (6., 4.5);
          \node[inner sep=0, anchor=west] at (6.6, 3.9) {$t_6(a_4)$};
          \node[inner sep=0] at (6.3, 4.2) {$\overline a_1$};

        \end{scope}

      \end{tikzpicture}
      \caption{Two possible configurations for the solution $u_1$ of \Cref{ex:Fabio}, 
        constructed with boundary control $f(k_1(t))$,
        where the boundary datum $k_1$ is given in~\eqref{eqbdry1-new}. At left the case
        $t_6(a_4) \ge t_6$. At right the case $t_6(a_4) < t_6$. 
        Here $t_6(a_4)$ is defined in~\eqref{prova11}, while $t_6$ is defined in~\eqref{t_6}.}
      \label{fig:ex-fabio-fig2-u1}
    \end{figure}

    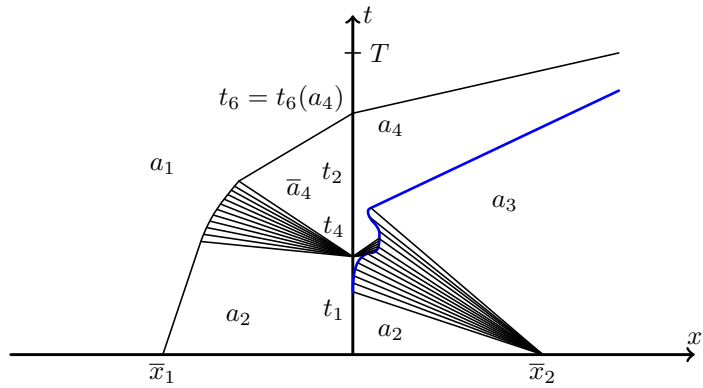
\begin{figure}
      \centering
      \begin{tikzpicture}[line cap=round,line join=round,x=1.cm,y=1.cm]
    
        \draw[->, line width=1.1pt, name path=vertical] (6.5, .5) -- (6.5, 5.) node[right]{$t$};
        \draw[->, line width=1.1pt] (2., .5) -- (11., .5) node[above]{$x$};

        \node[inner sep=0, anchor=north] at (4., 0.4) {$\overline{x}_1$};
        \node[inner sep=0, anchor=north] at (9, 0.4) {$\overline{x}_2$};

        \draw[line width=.6pt] (4., 0.5) -- (4.5, 2) to [out=70, in=230] (5., 2.8) -- (6.5, 3.7) -- (10., 4.5);
        \draw[line width=.6pt, name path=rarefaction, draw=none] (9, 0.5) -- (4.5, 2) ;

        \path[name intersections={of=vertical and rarefaction, by=t1}];

        \draw[line width=.6pt, draw=none, name path=raradue] (9, 0.5) -- (4., 4.8);
        \path[name intersections={of=vertical and raradue, by=t2}];
        \node[anchor=south east] at (t2) {$t_2$};

        \draw[line width=1.pt, color=blue] (t1) to [out=90, in=185] (6.75, 1.85) to [out=5, in=-90] (6.85,2.1) 
        to [out=90, in=-90] (6.7, 2.4) to [out=90, in=200] (6.75, 2.45) to (10, 4.);

        \begin{scope}
          \clip (t1) to [out=90, in=185] (6.75, 1.85) to [out=5, in=-90] (6.85,2.1) 
          to [out=90, in=-90] (6.7, 2.4) to [out=90, in=200] (6.75, 2.45) to (10, 4.) -- (10.5, 0.5) -- (6.5, 0.5) -- cycle;
          \draw[line width=.6pt] (9, 0.5) -- (4., 2.4) ;
          \draw[line width=.6pt] (9, 0.5) -- (4., 2.6) ;
          \draw[line width=.6pt] (9, 0.5) -- (4., 2.8) ;
          \draw[line width=.6pt] (9, 0.5) -- (4., 3.) ;
          \draw[line width=.6pt] (9, 0.5) -- (4., 3.2) ;
          \draw[line width=.6pt] (9, 0.5) -- (4., 3.4) ;
          \draw[line width=.6pt] (9, 0.5) -- (4., 3.6) ;
          \draw[line width=.6pt] (9, 0.5) -- (4., 3.8) ;
          \draw[line width=.6pt] (9, 0.5) -- (5.5, 3.);
          \draw[line width=.6pt] (9, 0.5) -- (4.5, 2) ;
          \draw[line width=.6pt] (9, 0.5) -- (4., 4.3) ;
          \draw[line width=.6pt] (9, 0.5) -- (4., 4.5) ;
          \draw[line width=.6pt] (9, 0.5) -- (4., 4.8) ;
        \end{scope}

        \begin{scope}
          \clip (t1) to [out=90, in=185] (6.75, 1.85) to [out=5, in=-90] (6.85,2.1) 
          to [out=90, in=-90] (6.7, 2.4) to [out=90, in=200] (6.75, 2.45) to (10, 4.) -- (6.5, 4.4) -- cycle;
          \draw[line width=.6pt] (6.5, 1.8) -- (8.5, 2.);
          \draw[line width=.6pt] (6.5, 1.8) -- (8.5, 2.2);
          \draw[line width=.6pt] (6.5, 1.8) -- (8.5, 2.4);
          \draw[line width=.6pt] (6.5, 1.8) -- (8.5, 2.6);
          \draw[line width=.6pt] (6.5, 1.8) -- (8.5, 2.8);
          \draw[line width=.6pt] (6.5, 1.8) -- (8.5, 3.);
          \draw[line width=.6pt] (6.5, 1.8) -- (8.5, 3.2);
        \end{scope}

        \draw[line width=.6pt] (6.5, 1.8) -- (4.5, 2);
        \draw[line width=.6pt] (6.5, 1.8) -- (5., 2.8);

        \node[inner sep=0] at (4, 3) {$a_1$};
        \node[inner sep=0] at (5, 1.) {$a_2$};
        \node[inner sep=0] at (7, .8) {$a_2$};
        \node[inner sep=0] at (8.5, 2.5) {$a_3$};
        \node[inner sep=0] at (5.8, 2.7) {$\overline a_4$};
        \node[inner sep=0] at (7, 3.5) {$a_4$};

        \begin{scope}
          \clip (6.5, 1.8) -- (4.5, 2) to [out=70, in=230] (5., 2.8) -- (6.5, 1.8);
          \draw[line width=.6pt] (6.5, 1.8) -- (4.5, 2.1);
          \draw[line width=.6pt] (6.5, 1.8) -- (4.5, 2.2);        
          \draw[line width=.6pt] (6.5, 1.8) -- (4.5, 2.3);
          \draw[line width=.6pt] (6.5, 1.8) -- (4.5, 2.4);
          \draw[line width=.6pt] (6.5, 1.8) -- (4.5, 2.5);
          \draw[line width=.6pt] (6.5, 1.8) -- (4.5, 2.6);
          \draw[line width=.6pt] (6.5, 1.8) -- (4.5, 2.7);
          \draw[line width=.6pt] (6.5, 1.8) -- (4.5, 2.8);
          \draw[line width=.6pt] (6.5, 1.8) -- (4.5, 2.9);
          \draw[line width=.6pt] (6.5, 1.8) -- (4.5, 3.);
        \end{scope}

        \node[inner sep=0, anchor=east] at (6.4, 1.1) {$t_1$};
        \node[inner sep=0, anchor=east] at (6.4, 2.2) {$t_4$};
        \node[inner sep=0, anchor=east] at (6.4, 3.9) {$t_6 = t_6(a_4)$};

        \draw[line width=.6pt] (6.4, 4.5) -- (6.6, 4.5) node[right]{$T$};

      \end{tikzpicture}
      \caption{The solution $u$ of \Cref{ex:Fabio}, constructed with boundary control $  f(k_1(t)) = f(k_2(t))$,
        where the boundary data $k_1$ and $k_2$ are given in~\eqref{eqbdry1-new}.}
      \label{fig:ex-fabio-fig2}
    \end{figure}

  Recalling that $t \mapsto f(\widetilde u(t,0))$ is continuously increasing for $t\in (t_1, t_2)$ between $f(a_2)$ and $f(a_3)$ and that $f(\widetilde u(t,0))\equiv f(a_3)$ for $t\in (t_2,t_3)$, we have that there exists  a unique $t_4\in (t_1,t_2)$ such that 
  \begin{equation}\label{tetto}
    \int_{t_1}^{t_3}f(\widetilde u(t,0))dt = (t_4-t_1)f(a_2)+(t_3-t_4)f(a_3).
  \end{equation}
  Note that for  $a\in (\overline a_2, \overline a_3)$ it holds that 
  \[
    \int_{t_1}^{t_3}f(\widetilde u(t,0))dt =(t_4-t_1)f(a_2)+(t_3-t_4)f(a_3)> (t_4-t_1)f(a_2)+ (t_3-t_4)f(a).
  \]
  For every 
  $a\in (\overline a_2, \overline a_3)$ 
  let us define
  $t_6(a)>t_3$ such that
  \begin{equation}\label{prova11}
    \int_{t_1}^{t_6(a)}
    f(\widetilde u(t, 0))dt =(t_4-t_1)f(a_2)+ (t_6(a)-t_4) f(a)
  \end{equation}
  see \Cref{fig7}.
  We observe that 
  \begin{equation*}
      \lim_{a\to \overline a_3}t_6(a)=t_3,
  \end{equation*}
  and so we   choose $a=a_4$ so to have 
  $t_6(a_4)\in (t_3, T)$.
  \begin{figure}
    \centering 

    \begin{tikzpicture}[scale=1.5] 
      \draw[->] (0,0) -- (7,0) node[right] {$t$};
      \draw[->] (0,0) -- (0,4) node[right] {$f$};
      
      \node[left] at (0,1) {\small{$f(a_1)$}};
      \node[left] at (0,2) {\small{$f(a_2) = f(\overline a_2)$}};
      \node[left] at (0,3.1) {\small{$f(a_4)=f(\overline a_4)$}};
      \node[left] at (0,3.5) {\small{$f(a_3)=f(\overline a_3)$}};
      
      \node[below] at (1,0) {$t_1$};
      \node[below] at (2.1,0) {$t_4$};
      \node[below] at (4,0) {$t_2$};
      \node[below] at (5,0) {$t_3$};
        \node[below] at (5.5,0) {$t_6(a_4)$};
      \node[below] at (6,0) {$T$};
      
      \draw[blue, thick] (1,2) -- (2.1,2);
      \draw[blue, thick] (2.2,3.25) -- (5.2,3.25);
      \draw[blue, thick] (5.2,1.05) -- (6,1.05);
      \draw[red, thick] (1,1.97) -- (2.1,1.97);
      \draw[red, thick] (2.1,3.4) -- (4,3.4);

      \draw[thick] (1,2) to[out=30,in=180] (4,3.38);
      \draw[thick] (4,3.38) -- (5,3.38);
      \draw[thick] (5,1) -- (6,1);
      
      \draw[dashed] (1,0) -- (1,2);
      \draw[red, dashed] (2.1,0) -- (2.1,3.4);
      \draw[dashed] (4,0) -- (4,3.5);
      \draw[dashed] (5,0) -- (5,3.5);
     \draw[dashed] (5.2,0) -- (5.2,3.3);
      \draw[dashed] (6,0) -- (6,1);

      \draw[dashed] (0,1) -- (6,1);
      \draw[dashed] (0,2) -- (1.9,2);
      \draw[dashed] (0,3.25) -- (2.5,3.25);
      \draw[dashed] (0,3.38) -- (5,3.38);

      \node at (1.8,2.20) {\small{\color{red}{$B_1$}}};
      \node at (2.5,3.1) {\small{\color{red}{{$B_2$}}}};

      \node at (1.5,2.18) {\small{\color{blue}{$A_1$}}};
      \node at (2.3,3) {\small{\color{blue}{$A_2$}}};
      \node at (4,3.35) {\small{\color{blue}{$A_3$}}};
      \node at (5.1,2) {\small{\color{blue}{$A_4$}}};
      \draw[black, thick] (0.5,-1) -- (2,-1) node[right] {$f(\widetilde{u}(t,0))$};
      \draw[blue, thick] (4,-1) -- (5.5,-1) node[right] {$f(k_1(t))=f(k_2(t))$};
    \end{tikzpicture}
    \caption{The construction related to the conditions~\eqref{tetto} and~\eqref{prova11}. 
      The time $t_4$ is such that the area $B_1$ is equal to the area $B_2$. 
      The value $f(a_4)$ and $t_6(a_4)$ are such that the area of $A_1+A_3$ 
      is equal to the area of $A_2+A_4$. }
    \label{fig7}
  \end{figure}
  Then define
\begin{equation}
    \label{eqbdry1-new}
    k_1(t) = \left\{
      \begin{array}{l@{\qquad}l}
        a_2,
        & t \in (0, t_4\,),
        \\
        \overline{a}_4,
        & t \in (t_4,
        t_6(a_4)),
        \\
       \overline a_1,
        & t \in (t_6(a_4), T),
      \end{array}
    \right.
    \qquad
    k_2(t) = \left\{
      \begin{array}{l@{\qquad}l}
         a_2,
        & t \in (0, t_1),\\
        \overline a_2,
        & t \in (t_1,   t_4\,),
        \\
         a_4,
        & t \in (  t_4, t_6(a_4)),
        \\
        a_1,
        & t \in (t_6(a_4), T).
      \end{array}
    \right.
  \end{equation} 
  Note that the solution $u_1$ has two possibile configurations, according to the fact that
  $t_6(a_4) \ge t_6$ or $t_6(a_4) < t_6$, 
  where $t_6(a_4)$ is defined in~\eqref{prova11} while $t_6$ in~\eqref{t_6}; 
  see \Cref{fig:ex-fabio-fig2-u1}. 
  By construction we have
  \begin{equation}
  \label{prop-t4a}
      \int_0^{t}
    f(\widetilde u(t,0))dt = 
    \int_0^{t}
    f(k_2(t))dt=
    \int_0^{t}
    f(k_1(t))dt
    \qquad \forall~t\in (0,t_1].
  \end{equation}
  and
  \begin{equation}
  \label{prop-t4ab}
      \int_0^{t}
    f(\widetilde u(t,0))dt > 
    \int_0^{t}
    f(k_2(t))dt=
    \int_0^{t}
    f(k_1(t))dt
    \qquad \forall~t\in (t_1,t_6(a_4)).
  \end{equation}

  We prove that $t_6(a_4) = t_6$ and~\eqref{claimfabio1}. Since $t_6(a_4) > t_3$ and using~\eqref{prova11} and the Divergence Theorem applied to the vector
  field $(\widetilde u, f(\widetilde u))$ on $(0, t_6(a_4)) \times (\overline x_1, 0)$, we 
  obtain
  \begin{align*}
    t_4 f(a_2) + (t_6(a_4) - t_4) f(a_4)
     = \int_{0}^{t_6(a_4)} f(\widetilde{u}(t, 0)) dt
    = t_6(a_4) f(a_1) + \overline{x}_1 (a_1 - a_2) 
  \end{align*}
  we directly obtain that $t_6(a_4) = t_6$,   
where $t_6$ is defined in~\eqref{t_6}, and then ~\eqref{claimfabio1},.
  Hence the structure of $u_1$ is that given in \Cref{fig:ex-fabio-fig2-u1}, left.

  We prove~\eqref{claimfabio2}.  
  Assume by contradiction that the shock generated at $t_1$ and entering $x>0$ hits again $x=0$ at a time $t^*>t_1$. 
  
If $t^* \le t_3$, then $u_2(t^*, x) = \widetilde u(t^*, x)$ for a.e. $x>0$.
  By \Cref{uniq}, we get that
  \[
    \int_{0}^{t^*}f(\widetilde u( t,0))dt = \int_{0}^{t^*}f( u_2( t,0))dt.
  \] 
  This condition, together with  \eqref{prova11}, \eqref{prop-t4a} and the fact that $f(u_2(t,0)) = f(k_2(t))$ 
  for $t\in (0, t^*)$ implies a contradiction.

  If $t^* > t_3$, then necessarily $u_2(t^*, 0)= a_3$. Moreover
  either $\lim_{t\to (t^*)^-} f(u_2(t, 0))=f(a_4)<f(a_3)$ or $\lim_{t\to (t^*)^-} f(u_2(t, 0))=f(a_1)<f(a_3)$.
  This condition implies that  the slope of the shock arriving in $t^*$ is positive, giving a contradiction.

We conclude with the following observation: the function $a \to t_6(a)$ is decreasing (see \Cref{fig7}) and so in particular there exists a 
value 
$\overline a_4\in (\overline a_2, \overline a_3)$ such that $t_6(\overline a_4)=T$. Note that $f(\overline a_4)$ is the minimum value among all possible values $f(a_4)$ associated with the previous construction. The   associated inflow admissible control $
\overline\gamma$ 
 satisfies the same properties as before, in particular it is a maximizer of \eqref{max-J},  and we have
\[
    \mathcal{F}_{[0,T]} (\overline\gamma) = \abs{f(a_2) - f(\overline a_4)} + \abs{f(\overline a_4) - f(a_1)}
    =  2f(\overline a_4)-f(a_1) -f(a_2).
  \] Note that 
  $\mathcal{F}_{[0,T]} (\overline\gamma)$ is the minimum  value achieved by $\mathcal{F}_{[0,T]} (\gamma)$ among all inflow admissible controls $\gamma$ associated to non entropic solutions 
  constructed with 
  boundary data of the type in~\eqref{eqbdry1-new}.
  
\end{proof}
\end{example}

\section{The entropy solution solves the min-max problem \texorpdfstring{\eqref{maxmin}}{} in the monotone case}
\label{entrsolvminmax}
In this Section we provide the main result of the paper, showing that if the initial datum $\overline{u}$ is 
 monotone, then the entropy weak solution to \eqref{eq:CP-classic} is a solution to the max-min problem \eqref{maxmin}. 
 In particular, no procedures as the ones used in previous Section \ref{sectionesempi} can be implemented.

We start with  two preliminary lemmata based on the Divergence Theorem.

\begin{lemma}
  \label{lemmamon} 
  \begin{enumerate}
    \item Assume $\widetilde u$ satisfies
    \begin{equation*}
      \widetilde u(T,x)\leq \theta\qquad \text{for a.e.  } x\in (-\infty, 0).
    \end{equation*}  
    Let    $u$ be a solution to the maximization problem \eqref{max-J}, associated to an admissible  
    boundary inflow control $\gamma$. Then $u(T, 0^-; \gamma)=\widetilde u(T, 0^-)$.

    \item Assume $\widetilde u$ satisfies
    \begin{equation*}
      \widetilde u(T,x)\geq \theta,\qquad \text{for a.e.  } x\in (0, +\infty).
    \end{equation*} 
      Let    $u$ be a solution to the maximization problem \eqref{max-J}, associated to an admissible  boundary 
      inflow control $\gamma$. Then $u(T, 0^+; \gamma)=\widetilde u(T, 0^+)$.  
\end{enumerate}
 
\end{lemma} 

\begin{proof} 
  We prove only case 1. since case 2. is completely analogous. 

  Let $u(t,x;\gamma)$ be  a solution to the maximization problem \eqref{max-J} associated with $\gamma$. 
   Since $\gamma$ is fixed, from now on we drop the dependence on $\gamma$, and we denote the solution as $u(t,x)$. 
  
  Let $L$ be such that $L> \max\{ \|f'(u)\|_{\LL\infty([0,T]\times (-\infty, 0))}, 
  \|f'(\widetilde u)\|_{\LL\infty([0,T]\times (-\infty, 0))}\}$. Then arguing exactly as in the proof of 
  \Cref{uniq}  we have that $f(u(t, -LT^+)) = f(\widetilde u(t, -LT^+))$ for a.e. $t\in [0,T]$. 
  
%

  Now we claim that 
  $u(T,x)=\widetilde u(T,x)$ for $x\in ( -LT,0)$.  
  Assume by contradiction that  there exists a continuity point $\overline{x}\in (-LT,0)$ for $u, \widetilde u$ such  
  that $u(T,\overline{x})\neq \widetilde u(T,\overline{x})$. 
  Since the initial datum for $u$ and $\widetilde u$ is the same, 
  there exists $\overline{t} \in (0,T)$ such that the 
  points $u(T, \overline{x})$ and $u(\overline{t},0)$ share the same generalized characteristic line. 
  Since the generalized characteristic lines cannot cross, for every $t\in [\overline{t}, T]$, the characteristic 
  passing through $(t,0)$ has negative speed. Therefore 
  for every $t\in [\overline{t}, T]$, we have $u(t,0)>\theta$, 
  and then also  $u(T,x)>\theta\geq \widetilde u(T,x)$ for every $x\in[\overline{x}, 0]$. 

  On the basis of the previous argument,  we define 
  \begin{equation*}
    \hat x:= \inf\{x\in [-LT,0]\ :\ u(T,x)\neq \widetilde u(T,x)\}= \inf\{x\in [-LT,0]\ :\ u(T,x)> \widetilde u(T,x)\}.
  \end{equation*}
  Therefore, by the previous argument we get that $u(T,x)>\theta\geq \widetilde u(T,x)$ for every $x\in (\hat{x}, 0]$. 
  Using the Divergence Theorem, 
  we get
  \begin{align*}
    \int_0^T f(u(t,0))dt
    & = \int_0^T f(u(t, -LT))dt + \int_{-LT}^0 u(0,x)dx -\int_{-LT}^0 u(T,x)dx
    \\
    & = \int_0^T f(\widetilde u(t, -LT))dt + \int_{-LT}^0 \overline u_1(x)dx -\int_{-LT}^0 u(T,x)dx
    \\
    & < \int_0^T f(\widetilde u(t, -LT))dt + \int_{-LT}^0 \overline u_1(x)dx -\int_{-LT}^0 \widetilde u(T,x)dx
    \\
    & = \int_0^T f(\widetilde u(t,0))dt, 
  \end{align*}
  that contradicts the fact that $\widetilde u$ is a solution to \eqref{max-J}.
  Therefore, we get that $u(T,x)=\widetilde u(T,x)$ for a.e. $x\in ( -\infty,0)$. 
  \end{proof} 

\begin{lemma}
  \label{lemmararefaction} 
  Assume $\widetilde u(T, \cdot)$ is monotone non-increasing.  
  Let    $u$ be a solution to the maximization problem \eqref{max-J}, associated to an admissible
  boundary inflow control $\gamma$. 
  Then $\widetilde{u}(T, \cdot)$ is continuous. Moreover,
  \begin{enumerate}
    \item if $\widetilde u(T, 0)\geq \theta $ we have  $u(T, 0^+; \gamma)=\widetilde u(T, 0)$;
    \item  if $\widetilde u(T, 0)\leq \theta $ we have  $u(T, 0^-; \gamma)=\widetilde u(T, 0)$.
  \end{enumerate} 
\end{lemma} 
\begin{proof} 
  Since $\widetilde{u}$ is the entropy admissible solution
  and $\widetilde{u}(T, \cdot)$ is non-increasing, then Oleinik type estimates imply that
  $\widetilde{u}(T, \cdot)$ is continuous. Moreover, 
  also the initial datum $\overline u$ is monotone non-increasing. 
  Indeed, due to the fact that  $\widetilde u(T,\cdot)$ is non-increasing, 
  for every $x\in \R$, there hold $\overline u(x^+)= \widetilde u(T, \xi^+(x))\leq  \widetilde u(T, \xi^-(x))=\overline u(x^-)$, 
  where $\xi^+(x), \xi^-(x)$ are respectively the intersection with $t=T$ of the maximal 
  and minimal characteristic issuing from $x$.

  We prove only case 1. since case 2. is completely analogous. Therefore, we assume that
  $\widetilde u(T, 0)\geq \theta$. Define $\bar x\in [0, +\infty]$ 
  as $\bar x = \inf \left\{x \in \R: \widetilde{u}(T, x) = \theta\right\}$ and denote
  with $u(t,x;\gamma)$ a solution to the maximization problem \eqref{max-J} associated with $\gamma$.

  We start with the case $\widetilde{u}(T, 0) = \theta$, so  that  $\bar x=0$. 
  In this case $f(\widetilde u(t,0)) = f(\theta)$ for every $t\in (0, T)$. 
  So $u$ necessarily satisfies $\int_0^T f(u(t, 0^+; \gamma))dt = f(\theta)T$, since it is a solution to \eqref{max-J}. 
  This immediately implies that $u(T, 0^+; \gamma) = \widetilde u(T, 0) = \theta$. \\
  Indeed, if $u(T, 0^+; \gamma) < \theta$, then there exists $\delta>0$ such that 
  $u(T, x; \gamma)<\theta$ for 
  $x\in (0, \delta)$, which implies that $\tau < T$, where
  $\tau=\inf\left\{t\in [0, T]: u(s, 0^+) < \theta \, \textrm{ for a.e. } s \in [t, T]\right\}$.
  This contradicts the fact that $f(u(t,0^+))=f(\theta)$ for a.e. $t \in [0, T]$. \\
  On the other hand it is also not possible that $u(T, 0^+; \gamma)>\theta$, since this would contradict 
  the fact that $u, \widetilde u$ share the same initial datum and $\widetilde u(T,0) = \theta$. 

  Assume now $\widetilde{u}(T, 0) > \theta$, so that $\bar x>0$, and suppose
  by contradiction that $u(T, 0^+; \gamma)\neq \widetilde u(T, 0)$.  
  Fix $L > \max \left\{\|f'(u)\|_{\LL\infty([0,T]\times (0, +\infty))}, 
  \|f'(\widetilde u)\|_{\LL\infty([0,T]\times (0,+\infty))}\right\}$.
  There exists $\delta>0$ such that $u(T, x; \gamma)\leq \theta $ for $x\in (0, \delta)$.
  Define $\tau=\inf\left\{t\in [0, T]: u(t, 0^+; \gamma)\leq \theta\right\}$. 
  Let $\xi$ be the forward characteristic issuing from $(\tau, 0)$ and 
  define $\hat x = \xi (T)$, so that $\hat x\in (0, LT]$.
  Then necessarily $u(T, x; \gamma) = \widetilde u(T,x)$ for $ x \in (\hat x, LT)$
  and $u(T, x; \gamma) \le \theta$ for $x \in (0, \hat x)$.  
  We claim that
  \begin{equation}
    \label{integrale}
    \int_0^{\hat x}  u(T,x; \gamma) dx < \int_0^{\hat x}  \widetilde u(T,x)dx.
  \end{equation} 
  In the case $\hat x \le \bar x$, we have $u(T,x)\leq \theta<\widetilde u(T,x)$ for a.e. $x \in (0, \hat x)$ 
  and so \eqref{integrale} holds.
  \\
  In the case $\hat x > \bar x$ we have that $u(T, x; \gamma) \le \theta$ and 
  $\widetilde{u}(T, x) \le \theta$ for a.e. $x \in (\bar x, \hat x)$.
  Moreover, if $x \in (\bar x, \hat x)$, then 
  $f'(u(T,x^+; \gamma )) \geq \frac{x}{T}$ and $f'(\widetilde u(T,x)) \leq \frac{x}{T}$;
  hence $u(T, x^+; \gamma) \le \widetilde{u}(T, x)$ and so \eqref{integrale} holds.


  Reasoning as in \Cref{lemmamon}, using of the Divergence Theorem and \eqref{integrale}, we deduce that
  \begin{equation*}
    \int_0^T f(\widetilde{u}(t, 0^+)) dt > \int_0^T f(u(t, 0; \gamma)) dt,
  \end{equation*}
  which contradicts the fact that $u(\cdot, \cdot; \gamma)$ is a solution to the maximization problem~\eqref{max-J}.  
  \end{proof} 
\subsection{Monotone non-decreasing initial datum}
In this section we   consider the case in which the initial datum may only generate shock waves or compression waves.

Note that since $\overline u$ is non-decreasing, there exists $\bar x\in[-\infty, +\infty]$   such that  
\begin{equation}\label{xbar} \bar x:= \sup\left\{x \in \R: \overline u(x) \le \theta\right\}.\end{equation}
  
We  start computing in the next two lemmata the functional $\mathcal{F}_{[0,T]}$ defined in \eqref{tv-def} 
for the entropy solution.

\begin{lemma}\label{lemmafdecrescente}
  Assume that $\overline u $ is monotone non-decreasing and let  $\widetilde u$ be the unique entropy solution 
  to~\eqref{eq:CP-classic}.
  Then the following holds:
\begin{enumerate} \item for every $\hat x \in \R$, the map $t \mapsto f(\widetilde{u}(t, \hat{x}))$ is
  non-increasing,
  \item $f(\widetilde{u}(T^-, 0))\geq f(\widetilde{u}(T, 0^+)), f(\widetilde{u}(T, 0^-))$.\end{enumerate}

In particular, denoting $\widetilde \gamma=f(\widetilde u(\cdot,0))$, we have that 
\begin{align*} \mathcal F_{[0, T]}(\widetilde\gamma)
     = &
      \tv_{(0, T)}(f(\widetilde u(\cdot, 0)))
    +
    |f(\overline u(0^-))- \widetilde\gamma(0^+)| + 
    |f(\overline u(0^+))- \widetilde \gamma(0^+)| \\
    &+ 2\widetilde\gamma(T^-)-f(\widetilde{u}(T, 0^+))-f(\widetilde{u}(T, 0^-))
\\ = &  \widetilde \gamma(0^+)- \widetilde \gamma(T^-) +
    |f(\overline u(0^-))- \widetilde\gamma(0^+)| + 
    |f(\overline u(0^+))- \widetilde \gamma(0^+)| \\
    &+ 2\widetilde\gamma(T^-)-f(\widetilde{u}(T, 0^+))-f(\widetilde{u}(T, 0^-)).\end{align*}  
\end{lemma}
\begin{proof}
We start proving that $t\to f(\widetilde u(t, \hat x))$ is non-increasing.   Without loss of generality we assume 
that $\hat{x} = 0$.

  By contradiction assume that there exist $0 < t_1 < t_2$ such that
  $f(\widetilde{u}(t_1, 0)) < f(\widetilde{u}(t_2, 0))$. We may assume without loss of generality that they are continuity points. 
  Note that the Lax condition implies that the map $t \mapsto f(\widetilde{u}(t, 0))$
  can have only downward jumps. Hence, we may assume that there exists  a continuity point $\bar t \in (t_1, t_2)$ such that 
      \begin{equation}
        \label{eq:intermediate-flux}
        f(\widetilde{u}(t_1, 0)) < f(\widetilde{u}(\bar t, 0)) < f(\widetilde{u}(t_2, 0)). 
      \end{equation}

  Denote with $\xi_1, \xi_2$ the backwards generalized characteristics exiting from
  $(t_1, 0)$ and $(t_2, 0)$. We have some possibilities.
  \begin{itemize}
      \item[(i)] $\xi_1(0) < \xi_2(0)$. In this case we necessarily deduce that $\xi_2(0) > 0$, otherwise
      $\xi_1$ and $\xi_2$ cross each other. Thus, $\widetilde{u}(t_2, 0) > \theta$.
      The monotonicity of $\overline u$ implies that
      $\widetilde{u}(t_1, 0) = \overline u(\xi_1(0)) \le \overline u(\xi_2(0)) = \widetilde{u}(t_2, 0)$.
      Since $f(\widetilde{u}(t_1, 0)) < f(\widetilde{u}(t_2, 0))$ we deduce that
      $\widetilde{u}(t_1, 0) < \theta$ and $\xi_1(0) < 0$.

      Denote with $\bar \xi$ the backward generalized characteristic exiting from $(\bar t, 0)$.\\
      If $\bar \xi(0) \le 0$, then $\bar \xi(0) < \xi_1(0)$ (otherwise $\bar \xi$ crosses $\xi_1$)
      and so $\widetilde{u}(\bar t, 0) \le \widetilde{u}(t_1, 0)<\theta $, 
      violating~\eqref{eq:intermediate-flux}.\\
      If $\bar \xi(0) > 0$, then $\bar \xi(0) \le \xi_2(0)$ (otherwise $\bar \xi$ crosses $\xi_2$)
      and so $\widetilde{u}(t_1, 0) < \theta < \widetilde{u}(\bar t, 0) \le \widetilde{u}(t_2, 0)$, 
      violating~\eqref{eq:intermediate-flux}.

      Therefore, this case is not possible.
    
    \item[(ii)] $\xi_1(0) = \xi_2(0)$. So in  $\xi_1(0)$ the initial data $\overline u$ would have a downward jump, 
    which is not possible by the assumption on monotonicity. 

    \item[(iii)] $\xi_1(0) > \xi_2(0)$. In this case we necessarily deduce that $\xi_2(0) < 0$, otherwise
      $\xi_1$ and $\xi_2$ cross each other. Thus, $\widetilde{u}(t_2, 0) < \theta$.
      The monotonicity of $\overline u$ implies that
      $\widetilde{u}(t_1, 0) = \overline u(\xi_1(0)) \ge \overline u(\xi_2(0)) = \widetilde{u}(t_2, 0)$.
      Since $f(\widetilde{u}(t_1, 0)) < f(\widetilde{u}(t_2, 0))$ we deduce that
      $\widetilde{u}(t_1, 0) > \theta$ and $\xi_1(0) > 0$.

      Consider now $\bar t \in (t_1, t_2)$ such that~\eqref{eq:intermediate-flux} holds 
      and $(\bar t, 0)$ point of continuity for $\widetilde{u}$.
      Denote with $\bar \xi$ the backward generalized characteristic exiting from $(\bar t, 0)$.\\
      If $\bar \xi(0) \ge 0$, then $\bar \xi(0) > \xi_1(0)$ (otherwise $\bar \xi$ crosses $\xi_1$)
      and so $\widetilde{u}(\bar t, 0) \ge \widetilde{u}(t_1, 0)$, 
      violating~\eqref{eq:intermediate-flux}.\\
      If $\bar \xi(0) < 0$, then $\bar \xi(0) \ge \xi_2(0)$ (otherwise $\bar \xi$ crosses $\xi_2$)
      and so $\widetilde{u}(t_2, 0) \le \widetilde{u}(\bar t, 0) < \theta <  \widetilde{u}(t_1, 0)$, 
      violating~\eqref{eq:intermediate-flux}.

      Therefore, this case is not possible. 
  \end{itemize}
So, the proof of 1. is concluded.
\smallskip

Now we show that $f(\widetilde{u}(T^-, 0))\geq f(\widetilde{u}(T, 0^+)), f(\widetilde{u}(T, 0^-))$. We show that 
$f(\widetilde{u}(T^-, 0))\geq  f(\widetilde{u}(T, 0^-))$ (the case $f(\widetilde{u}(T^-, 0))\geq f(\widetilde{u}(T, 0^+))$ 
is completely analogous). 

Let us consider an increasing  sequence $x_n\to 0^-$   such that $(T,x_n)$  are continuity points for $\widetilde u$ 
and $f(\widetilde{u}(T, x_n))\to f(\widetilde{u}(T, 0^-))$. We denote with $\xi_n$ the backward genuine characteristic 
passing through $(T,x_n)$.

We distinguish three cases.
\begin{itemize}
\item[(i)] There exists $\bar t<T$ such that all points $(t,0)$ for $t\in (\bar t, T)$ 
are discontinuity points for $\widetilde u$. In this case there exists a shock located in $x=0$ in the interval 
$(\bar t, T)$, so $\widetilde u(t, 0^-)\leq \theta$ and $\widetilde u(t, 0^+)\geq \theta$ for all $t\in  (\bar t, T)$. 
We consider a sequence $t_n\to T^-$ and the backward minimal characteristics $\widetilde \xi_n$ from $(t_n, 0)$.  
Then necessarily $\xi_n(0)\leq \widetilde \xi_n(0)\leq \bar x$, where $\bar x$ is defined in \eqref{xbar}. 
Then, we conclude by the monotonicity of $\overline u$, we get that 
$\widetilde u(T,x_n)\leq \widetilde u(t_n, 0^-)\leq \theta$, from which we deduce 
 $ f(\widetilde{u}(T, 0^-))\leq f(\widetilde{u}(T^-, 0))$.
\item[(ii)] There exists $\bar t<T$ such that  there exists a dense subset  of $ (\bar t, T)x\{0\}$ 
of continuity points for $\widetilde u$. 

Assume that  
$\widetilde u(t, 0)\leq \theta$ for $t\to T^-$ and $(t,0) $ continuity point. Let us denote $\widetilde\xi$ is the backward 
genuine characteristic passing through $(t,0)$.

In this case, $\widetilde u(T, x_n)\leq \theta$, and we proceed exactly as in case $1$. 
\item[(iii)] There exists $\bar t<T$ such that  there exists a dense subset  of $ (\bar t, T)\times\{0\}$ 
of continuity points for $\widetilde u$. 

Assume that  
$\widetilde u(t, 0)> \theta$ for $t\to T^-$ and $(t,0) $  continuity point. Let us denote $\widetilde\xi$ 
is the backward genuine characteristic passing through $(t,0)$, for $t\in (\bar t, T)$, such that $(t,0)$ continuity point.

If $\widetilde u(T, x_n)>  \theta$, then we proceed as in case $1$, by symmetry.

It remains to consider the case $\widetilde u(T, x_n)\leq \theta$. Assume that there exists $\tau\in (t, T)$ 
such that $\widetilde \xi(\tau)=0$. Then this would imply $f'(\widetilde(\tau, 0^-))<\widetilde\xi'(\tau)$, 
in contradiction with Lax conditions. Moreover, due to the fact that $\widetilde u(T, 0^-)\leq \theta$, 
we conclude that necessarily $\widetilde \xi(T)=0$ with $\widetilde \xi'(T)\geq 0$. This means that $\widetilde{u}$ 
admits a shock discontinuity
at $(T,0)$, connecting the left state $\widetilde{u}(T, 0^-)$ with the right state $\widetilde{u}(T^-, 0)$,
that has slope $\widetilde \xi'(T)\geq 0$,
which implies that $f(\widetilde{u}(T^-, 0))\geq  f(\widetilde{u}(T, 0^-))$.
\end{itemize}

\end{proof}

\begin{lemma} 
  \label{lemmashock}
  Assume that $\overline u $  is  monotone non-decreasing and let  $\bar x\in[-\infty, +\infty]$ as in \eqref{xbar}.
  
Let $\widetilde u$ be the unique entropy solution to~\eqref{eq:CP-classic} and $\widetilde \gamma= f(\widetilde u(\cdot, 0))$. 
We denote $\bar \xi$ the unique forward generalized characteristic issuing from  $\bar x$ at time $0$.
  
We get 
  \begin{equation}
    \label{funzionalecasoshock}
    \mathcal F_{[0,T]}(\widetilde \gamma)=
    \begin{cases} \max(f(\overline u(0^+), f(\overline u(0^-)) - f(\widetilde u(T,0^-)),
      & 
     \textrm{if }   \bar x\in [0, +\infty] \textrm{ and }\bar \xi(T)> 0,
      \vspace{.2cm}
      \\ \max(f(\overline u(0^+), f(\overline u(0^-)) - f(\widetilde u(T,0^+)),
      & 
     \textrm{if }   \bar x\in [0, +\infty) \textrm{ and }\bar \xi(T)< 0,
      \vspace{.2cm}
      \\\vspace{.2cm}
      \max(f(\overline u(0^+), f(\overline u(0^-)) - f(\widetilde u(T,0^+)),
      & \textrm{if }   \bar x\in [-\infty,0] \textrm{ and }\bar \xi(T)< 0,    \\
      \max(f(\overline u(0^+), f(\overline u(0^-)) - f(\widetilde u(T,0^-)),

      & 
   \textrm{if }   \bar x\in (-\infty,0] \textrm{ and }\bar \xi(T)> 0,
      
    \end{cases} 
  \end{equation} 
  and finally if $\bar \xi(T)= 0$,
  \begin{equation}\label{zero}\mathcal F_{[0,T]}(\widetilde \gamma)\leq f(\overline u(0^+)) - f(\widetilde u(T,0^-)) +f(\overline u(0^-)) 
    - f(\widetilde u(T,0^+)).\end{equation}
  \end{lemma} 
  \begin{proof}

  Throughout the proof we rely on the properties of generalized characteristics, see \cite{d}. 
  Note that $f(\widetilde u(t, 0^-))= f(\widetilde u(t, 0^+))$ for a.e. $t$ and 
  $f(\widetilde u(\cdot, 0^-)),f(\widetilde u(\cdot, 0^+))\in \mathbf{BV}(0,T)$.
  
By the definition  \eqref{tv-def}, and Lemma \ref{lemmafdecrescente} we have that 
  \begin{align*}
    \mathcal F_{[0, T]}(\widetilde\gamma)
    & = 
     \widetilde \gamma(0^+)- \widetilde \gamma(T^-)
    +
    |f(\overline u(0^-))- \widetilde\gamma(0^+)| + |f(\overline u(0^+))- \widetilde \gamma(0^+)| 
    \\ 
    \nonumber 
    & \quad  +2 \widetilde\gamma(T^-)
    -f(\widetilde u(T,0^+))- f(\widetilde u(T,0^-))\\
     \nonumber 
    & = \widetilde \gamma(0^+)+\widetilde \gamma(T^-)+
    |f(\overline u(0^-))- \widetilde\gamma(0^+)| + |f(\overline u(0^+))- \widetilde \gamma(0^+)| 
    \\ 
    \nonumber 
    & \quad -f(\widetilde u(T,0^+))- f(\widetilde u(T,0^-)).
  \end{align*}
  
We consider different cases depending on the value of $\bar x$.  

  \begin{enumerate}
  \item  {\bf Either $\bar x=+\infty$, or $\bar x\geq  0$ and $\bar \xi(T)>0$.}

     Observe that 
    $\widetilde u(T,x)\leq \theta$, for a.e. $x\leq 0$. 

   We start considering $\bar x>0$. Since   $\overline u$ is monotone non-decreasing and $\overline u(x) \le \theta$ 
   for a.e. $x\leq \bar x$, we have that $f(\overline u(0^-))=\widetilde \gamma(0^+)$ and 
    $ \widetilde \gamma(0^+)=f(\overline u(0^-)) \le f(\overline u(0^+))$. 
    Note also that $f(\widetilde u(T,0^-))\leq \widetilde\gamma(T^-) = f(\widetilde u(T,0^+))$.
    Hence, we deduce that
    \begin{align*}
      \mathcal F_{[0, T]}(\widetilde\gamma)
      & = \widetilde\gamma(0^+)-\widetilde \gamma(T^-)
      + \widetilde\gamma(T^-)-f(\widetilde u(T,0^-))
      + f(\overline u(0^+))- \widetilde \gamma(0^+)
      \\
      & = f(\overline u(0^+))-f(\widetilde u(T,0^-)),
    \end{align*}
    proving~\eqref{funzionalecasoshock}.
  
For $\bar x=0$, we observe that   $\widetilde \gamma(0^+)=\min(f(\overline u(0^-), f(\overline u(0^+)) $.  
Then we may proceed as before, getting
\begin{align*}
      \mathcal F_{[0, T]}(\widetilde\gamma)
      & = \widetilde\gamma(0^+)-\widetilde \gamma(T^-)
      + \widetilde\gamma(T^-)-f(\widetilde u(T,0^-))
      + \max(f(\overline u(0^+), f(\overline u(0^-))- \widetilde \gamma(0^+)
      \\
      & = \max(f(\overline u(0^+), f(\overline u(0^-))-f(\widetilde u(T,0^-)).
    \end{align*}

  \item  {\bf $\bar x \geq 0$ and $\bar \xi(T)<0$.}

Note that in this case  $\widetilde u(T,x)\geq \theta$ for a.e. $x\geq 0$. 

 We get as in case $1$  we have that   $ \widetilde\gamma(0^+)=\min(f(\overline u(0^-)),   f(\overline u(0^+))$ 
 whereas  $f(\widetilde u(T,0^+)) \leq   \widetilde\gamma(T^-)=f(\widetilde u(T,0^-)) $. In conclusion
\begin{equation*}
   \mathcal F_{[0,T]}(\widetilde \gamma)  
    =  \max(f(\overline u(0^-)),   f(\overline u(0^+))- f(\widetilde u(T,0^+)) .
  \end{equation*}

  \item  {\bf Either $\bar x=-\infty$ or $\bar x\leq 0$ and $\bar \xi(T)<0$.}
  
  In this  case  we observe that  $u(T,x)\geq \theta$ for a.e. $x\geq 0$. We proceed as in  item $1$ by symmetry.
   
  \item  {\bf $ \bar x\leq  0$ and $\bar \xi(T)>0$.}

Note that in this case  $\widetilde u(T,x)\leq \theta$ for a.e. $x\leq 0$.  We proceed as in item $2$ by symmetry.

  \item  {\bf   $\bar \xi(T)=0$.}

Note that in this case  $\widetilde u(T,x)\geq \theta$ for a.e. $x\geq 0$ and $\widetilde u(T,x)\leq \theta$ for a.e. $x\leq 0$. 

Arguing as above, we have that   $ \widetilde \gamma(0^+)=\min(f(\overline u(0^-)), f(\overline u(0^+)))$.

So we get 
  \begin{eqnarray*}
   \mathcal F_{[0,T]}(\widetilde \gamma)  
    &=&  \widetilde\gamma(0^+)-\widetilde\gamma(T^-)+2\widetilde\gamma(T^-)- f(\widetilde u(T,0^-)) -f(\widetilde u(T,0^+))\\ &+& \max( f(\overline u(0^+),f(\overline u(0^-))-\widetilde\gamma(0^+)) \\ &=&  \max( f(\overline u(0^+),f(\overline u(0^-))+\widetilde\gamma(T^-))- f(\widetilde u(T,0^-))  - f(\widetilde u(T,0^+))
    \\ &\leq &  \max( f(\overline u(0^+),f(\overline u(0^-))+\widetilde\gamma(0^+)- f(\widetilde u(T,0^-))  - f(\widetilde u(T,0^+))\\ &=&    f(\overline u(0^+)+f(\overline u(0^-) - f(\widetilde u(T,0^-))  - f(\widetilde u(T,0^+)).
  \end{eqnarray*}
\end{enumerate}
  \end{proof}

  Based on the previous lemma we show that in the case the initial datum generates only shock waves, 
  the entropy solution solves the minimization problem \eqref{maxmin}. 
  \begin{theorem} 
    \label{mon} Assume that 
  $\overline u $ is  monotone non-decreasing. 
  Let $\widetilde u$ be the unique entropy solution to \eqref{eq:CP-classic}.
  Then $\widetilde u$ is a  solution to \eqref{max-J} and also of \eqref{maxmin} for all 
  $M\geq M_0:=\mathcal F_{[0,T]} \{f(\widetilde u(\cdot, 0))\}$. In particular
    \begin{equation*}\min_{\gamma\in \mathcal{U}^M_{max}}\mathcal{F}_{[0,T]}(\gamma)
    =\mathcal{F}_{[0,T]}(\widetilde \gamma ).
    \end{equation*}
  \end{theorem}
  \begin{proof}
    For every $\gamma\in \mathcal{U}^M_{max}$, we consider the associated function $u(t,x;\gamma)$, 
    see \Cref{defiammissibile}.  
    
    We fix $\gamma\in \mathcal{U}^M_{max}$, and we drop the dependence on $\gamma$ on the function $u$.
    Recalling \eqref{tv-def}, we get that 
  \begin{equation}\label{contomoduli} \begin{cases}\mathcal{F}_{[0,T]}(\gamma)\geq |f(\overline u(0^+)) - f( u(T,0^-))| 
    + |f(\overline u(0^-))-f(u(T,0^+))|,
\\   \mathcal{F}_{[0,T]}(\gamma)\geq  |f(\overline u(0^+)) - f( u(T,0^+))| + |f(\overline u(0^-))-f(u(T,0^-))|,\\   
\mathcal{F}_{[0,T]}(\gamma) \geq|\max(f(\overline u(0^+)),f(\overline u(0^-)))  - f( u(T,0^+))|,  \\ 
\mathcal{F}_{[0,T]}(\gamma)\geq |\max(f(\overline u(0^+)),f(\overline u(0^-)))  - f( u(T,0^-))|  . 
  \end{cases}\end{equation}

  We consider $\bar x$ as in \eqref{xbar}, and we denote $\bar \xi$ the unique 
  forward generalized characteristic issuing from  $\bar x$ at time $0$. 
  We divide the proof in several cases, using the same notation of \Cref{lemmashock}.  
  \begin{enumerate} 
 \item  {\bf Either $\bar x=+\infty$, or $\bar x\geq  0$ and $\bar \xi(T)>0$.}
 
Since $\widetilde u(T, x)\leq \theta $ for a.e.  $x\leq  0$, by \Cref{lemmamon}, we get that $f( u(T,0^-))=f( \widetilde u(T,0^-))$. So  recalling inequality  \eqref{contomoduli}, and by \eqref{funzionalecasoshock} we conclude that
\begin{eqnarray*}\mathcal{F}_{[0,T]}(\gamma)&\geq& |\max(f(\overline u(0^+), f(\overline u(0^-)) - f( u(T,0^-))| \\ &=&|\max(f(\overline u(0^+), f(\overline u(0^-)) - f( \widetilde u(T,0^-))|= \mathcal{F}_{[0,T]}(\widetilde \gamma) .\end{eqnarray*}

\item  {\bf   $\bar x\geq  0$ and $\bar \xi(T)<0$.}
 
Since $\widetilde u(T, x)\geq \theta $ for a.e.  $x\geq  0$, by \Cref{lemmamon}, we get that $f( u(T,0^+))=f( \widetilde u(T,0^+))$. So recalling inequality  \eqref{contomoduli}, and by \eqref{funzionalecasoshock} we conclude that
\begin{eqnarray*}\mathcal{F}_{[0,T]}(\gamma)&\geq& |\max(f(\overline u(0^+), f(\overline u(0^-)) - f( u(T,0^+))| \\ &=&|\max(f(\overline u(0^+), f(\overline u(0^-)) - f( \widetilde u(T,0^+))|= \mathcal{F}_{[0,T]}(\widetilde \gamma) .\end{eqnarray*}
\item  {\bf Either  $\bar x=-\infty$ or $\bar x\leq  0$ and $\bar \xi(T)<0$.}

It is the symmetric case to case $1$.
\item  {\bf   $\bar x\leq  0$ and $\bar \xi(T)>0$.}
 
It is the symmetric case to case $2$.

 \item  {\bf $\xi(T)=0$.} 

 In this case    $\widetilde u(x,T)\leq \theta$ a.e. for $x<0$ and $\widetilde u(x,T)\geq \theta$ a.e. for $x>0$. 
 Therefore, by \Cref{lemmamon} we get $f( u(T,0^-))=f( \widetilde u(T,0^-))$ and $f( u(T,0^+))=f( \widetilde u(T,0^+))$. 
 Recalling \eqref{zero} and \eqref{contomoduli}
 we get \begin{eqnarray*}
   \mathcal F_{[0,T]}(\widetilde \gamma)  
  &\leq & f(\overline u(0^+)) - f(\widetilde u(T,0^-))+f(\overline u(0^-)) - f(\widetilde u(T,0^+))\\&\leq&   |f(\overline u(0^+)) - f( u(T,0^-))| + |f(\overline u(0^-))-f(u(T,0^+))|\leq \mathcal{F}_{[0,T]}(\gamma).\end{eqnarray*} 
 \end{enumerate}
  \end{proof}
  
\subsection{Monotone non-increasing initial datum}
In this section we consider the case in which the initial datum may only generate rarefaction  waves.

Note that since $\overline u$ is non-increasing,   there exists $\bar x\in[-\infty, +\infty]$   
such that  \begin{equation}\label{xbar2} \bar x:= \sup\left\{x \in \R: \overline u(x) \ge \theta\right\}.\end{equation}

As in the previous section, we 
compute   $\mathcal{F}_{[0,T]}$ for the entropy solution $\widetilde u$ in this setting.
We  start computing in the next two lemmata the functional $\mathcal{F}_{[0,T]}$ for the entropy solution.

\begin{lemma}\label{lemmafcrescente}
  Assume   $\overline u $ is monotone non-increasing and let us consider   $\widetilde u$   the unique entropy solution to~\eqref{eq:CP-classic}.
  
  Then, for every $t > 0$, 
  the function $\widetilde{u}(t, \cdot)$ is continuous and for every $\hat x \in \R$, 
  the map $t \mapsto f(\widetilde{u}(t, \hat{x}))$ is
  non-decreasing.
  
In particular, denoting $\widetilde\gamma=f(\widetilde u(\cdot, 0))$, we have that 
\begin{eqnarray*} \mathcal F_{[0, T]}(\widetilde\gamma)
    & = &
      \tv_{(0, T)}(f(\widetilde u(\cdot, 0)))
    +
    |f(\overline u(0^-))- \widetilde\gamma(0^+)| + 
    |f(\overline u(0^+))- \widetilde \gamma(0^+)| 
\\ &= &  \widetilde \gamma(T)- \widetilde \gamma(0) +
    |f(\overline u(0^-))- \widetilde\gamma(0^+)| + 
    |f(\overline u(0^+))- \widetilde \gamma(0^+)| . \end{eqnarray*}  
     
\end{lemma}

\begin{proof}As before, we 
  rely on the properties of generalized characteristics, see \cite{d}.

We show that 
  the function $\widetilde{u}(t, \cdot)$ is continuous. 
  
  We fix $t=T$, but the argument is completely analogous for a generic $t$. 
  Assume by contradiction that there exists $\widetilde  x \in \R$ such that
  $\widetilde{u}(T, \widetilde x^-) \ne \widetilde{u}(T, \widetilde x^+)$. Denote with $\xi^-$ and
  $\xi^+$ respectively the minimal and the maximal backward characteristics from $(T, \widetilde  x)$. 
  Since  $\xi^-(0) < \xi^+(0)$
  and $\overline u$ non-increasing, we get $\overline u (\xi^-(0)^+) \geq \overline u(\xi^+(0)^-)$.
  This is not compatible with the fact that    $\dot \xi^- (t) = f'(\widetilde{u}(T, \bar x^-))\leq \dot \xi^+ (t) = f'(\widetilde{u}(T, \bar x^+))$ if $ \overline u (\xi^-(0)^+)\leq \theta$ and with the fact that $\dot \xi^- (t) = f'(\widetilde{u}(T, \bar x^-))\geq \dot \xi^+ (t) = f'(\widetilde{u}(T, \bar x^+))$ if $ \overline u (\xi^+(0)^-)\geq \theta$
  (see~\cite[(11.1.13) and Theorem~11.1.3]{daf-book}).  

We prove now that the map $t\to f(\widetilde{u}(t, \hat x))$ is non-decreasing. 
 Without loss of generality we assume that $\hat{x} = 0$.
 By contradiction assume that there exist $0 < t_1 < t_2\leq T$ such that
  $f(\widetilde{u}(t_1, 0)) >f(\widetilde{u}(t_2, 0))$.

  Denote with $\xi_1, \xi_2$ the genuine backward characteristics   from $(t_1, 0)$ and $(t_2, 0)$.
  \begin{enumerate}
 \item   If $\xi_1(0) = \xi_2(0)$, then a rarefaction wave is generated at $\xi_1(0)$. 
  In particular  $\widetilde u(t_1,0)=(f')^{-1}(-\xi_1(0) /t_1)$ 
  and $\widetilde u(t_2,0)=(f')^{-1}(-\xi_1(0) /t_2)$ and the traces
  $\widetilde u(t_1,0),\widetilde u(t_2,0)$ belong to $[\overline u (\xi_1(0)^+), \overline u(\xi_1(0)^-)]$.
  
  If $\xi_1(0)\leq 0$, we have that  $\overline u(\xi_1(0)^-)\leq \theta$ and then  $\widetilde u(t_1,0),\widetilde u(t_2,0)\leq \theta$. Using the fact that  $f$ is strictly concave, we get that $\widetilde u(t_1,0^-) = (f')^{-1}(-\xi_1(0)/t_1)
  < (f')^{-1}(-\xi_1(0)/t_2) = \widetilde u(t_2,0^-)$, and so we conclude  that $f(\widetilde u(t_1,0))< f(\widetilde u(t_2,0))$, in contradiction with our assumption. 

  If $\xi_1(0)>0$, we have that $\widetilde u(t_1,0),\widetilde u(t_2,0)\geq \theta$, and we argue exactly as before, by symmetry and obtain  $f(\widetilde u(t_1,0))< f(\widetilde u(t_2,0))$, in contradiction with our assumption.

  So, this case is not possible. 
  
 \item  If $\xi_1(0)>\xi_2(0)$, then, necessarily $\xi_2(0)\leq 0$,  and $\widetilde  u(t_2,0)\leq \theta$.  
 Moreover by monotonicity of $\overline u$, we have that  
  $\widetilde u(t_1,0) \leq \overline u(\xi_1(0)^-) \leq \overline u(\xi_2(0)^+)\leq \widetilde u(t_2,0)\leq \theta$. Therefore,    
  we conclude that $f(\widetilde u(t_1,0^-))\leq f(\widetilde u(t_2,0^-))$, in contradiction with our assumption.

  \item  If $\xi_1(0)<\xi_2(0)$, then, necessarily $\xi_1(0)\geq 0$,  and $\widetilde  u(t_1,0),\widetilde  u(t_2,0)\geq \theta$. 
   Recalling the monotonicity of $\overline u$, we get 
  $\widetilde u(t_1,0) \geq \overline u(\xi_1(0)^+) \geq \overline u(\xi_2(0)^-)\geq \widetilde u(t_2,0)\geq \theta$.  So   
  we conclude that $f(\widetilde u(t_1,0^-))\leq f(\widetilde u(t_2,0^-))$, in contradiction with our assumption.

      \end{enumerate}
      So, none of the previous cases is possible, and the proof is concluded. 
\end{proof}
\begin{lemma} 
  \label{monentro2}
  Assume that  
  $\overline u $ is monotone non-increasing. 
  Let $\widetilde u$ be the unique entropy solution to~\eqref{eq:CP-classic} and let 
  $\widetilde \gamma = f(\widetilde u(\cdot,0^-))$. 
  Then
  \begin{equation}
    \label{eq:estimate-F}
    \mathcal F_{[0,T]}(\widetilde \gamma)
    = 
    \begin{cases} 
       f(\widetilde u(T,0)) - \min\{f(\overline u(0^+)),f(\overline u(0^-))\}, 
      & \textrm{if } \bar x\neq 0,
      \\
      2 f(\theta) - f(\overline u(0^+)) - f(\overline u(0^-)),&  \textrm{if } \bar x=0,
    \end{cases} 
  \end{equation}
 where $\bar x$ has been defined in \eqref{xbar2}.

\end{lemma}

\begin{proof} 
By Lemma \ref{lemmafdecrescente} we have that
\[\mathcal F_{[0,T]}(\widetilde \gamma)=f( \widetilde u(T,0))- \widetilde \gamma(0^+) +
    |f(\overline u(0^-))- \widetilde\gamma(0^+)| + 
    |f(\overline u(0^+))- \widetilde \gamma(0^+)| .   \]

We consider different cases depending on the value of $\bar x$ as defined in \eqref{xbar2}.    

\begin{enumerate}
\item {\bf $\bar x=0$. }

In this case a rarefaction wave is generated at $0$, and $\widetilde u(t, 0)=\theta$ for $t>0$, so $f( \widetilde u(T,0))= \widetilde \gamma(0)=\theta$. 
Then 
\[\mathcal F_{[0,T]}(\widetilde \gamma)=\theta-\theta +
   \theta- f(\overline u(0^-))+\theta-  
    f(\overline u(0^+)),   \] which give \eqref{eq:estimate-F}.
    
  \item  {\bf   $\bar x<0$.} 

In this case $\overline u(0^+)\leq \overline u(0^-)\leq \theta$. Therefore, $\widetilde\gamma(0^+)=f(\overline u(0^-)) \geq f(\overline u(0^+)) $.

  Therefore, 
  \begin{equation*}
    \mathcal F_{[0,T]}  (\widetilde \gamma) = f( \widetilde u(T,0))- \widetilde \gamma(0^+) +
    \widetilde\gamma(0^+)- 
    f(\overline u(0^+)) = f( \widetilde u(T,0))- f(\overline u(0^+)).
  \end{equation*}   
 
   \item  {\bf   $\bar x>0$.} 

In this case $\theta\leq \overline u(0^+)\leq \overline u(0^-)$. Therefore, $\widetilde\gamma(0^+)=f(\overline u(0^+)) \geq f(\overline u(0^-)) $ 
and we conclude as above.   
\end{enumerate}
\end{proof} 

 Based on the previous lemmata we show that  also in the case the initial datum generates only rarefaction waves, 
 the entropy solution solves the minimization problem \eqref{maxmin}. 
  \begin{theorem} 
    \label{mon2} Assume that 
  $\overline u $ is  monotone non-increasing. 
  Let $\widetilde u$ be the unique entropy solution to \eqref{eq:CP-classic}.
  Then $\widetilde u$ is a  solution to \eqref{max-J} and also of \eqref{maxmin} for all $M\geq M_0:=\mathcal F_{[0,T]} \{f(\widetilde u(\cdot, 0))\}$. In particular
    \begin{equation*}\min_{\gamma\in \mathcal{U}^M_{max}}\mathcal{F}_{[0,T]}(\gamma)
    =\mathcal{F}_{[0,T]}(\widetilde \gamma ).
    \end{equation*}
  \end{theorem}
  \begin{proof}
    For every $\gamma\in \mathcal{U}^M_{max}$, we consider the associated function $u(t,x;\gamma)$, 
    see \Cref{defiammissibile}.  
    
     We fix $\gamma\in \mathcal{U}^M_{max}$, and we drop the dependence on $\gamma$ on the function $u$. 
     Recalling the definition of $\mathcal{F}_T$ in \eqref{tv-def}, we have that
       \begin{equation}\label{moduli2}
       \begin{cases} \mathcal{F}_{[0,T]}(\gamma) \geq |f( u(T,0^+))-\min(f(\overline u(0^+)),f(\overline u(0^-)))  |,  \\ 
       \mathcal{F}_{[0,T]}(\gamma)\geq | f( u(T,0^-))-\min f(\overline u(0^+)),f(\overline u(0^-)))   |  . 
  \end{cases}\end{equation}

We proceed with the different cases, depending on   $\bar x$, defined in \eqref{xbar2}, as in \Cref{monentro2}. 

\begin{enumerate}
\item {\bf $\bar x=0$}. 

We observe that if $\bar x=0$, then  $\widetilde u(0, t) =\theta$ for  $t\in (0,T)$. 
This implies that $\int_0^T f(\widetilde u(t,0))dt=f(\theta)T=(\max f) T$, therefore every other solution $u$ to the maximization problem \eqref{max-J} necessarily coincides with $\widetilde u$.

\item {\bf $\bar x\neq 0$. }

Note that due to the fact that $\overline u$ is monotone non-increasing, then only rarefaction waves can be generated, and $\widetilde u(T, \cdot)$ remains monotone non-increasing. 
We may apply then Lemma \ref{lemmararefaction}, which implies that either $f(\widetilde u(T,0))=f(u(T, 0^+))$ or  $f(\widetilde u(T,0))=f(u(T, 0^-))$. We conclude then using \eqref{moduli2} and  \eqref{eq:estimate-F} in Lemma \ref{monentro2}. 
\end{enumerate}
\end{proof}

\appendix
\section{Appendix}\label{appendix}
Throughout the section, we 
denote by $u(t,0)=u(t,0^+)$ the trace of the
solution of the initial-boundary value problem at $x=0$.
\begin{theorem}\label{bvfluxbdrytrace}
Let $f:\R\to \R$ be
of class $\CC1$, strictly concave, 
and satisfying~\eqref{def-theta}.
Given 
$\overline u\in \mathbf{BV} ((0,+\infty); \R)$,
and $k \in \mathbf{L^\infty}\left((0,+\infty)\,; (-\infty,\theta]\right)$, 
assume that $f\circ k\in \mathbf{BV} ((0,+\infty); \R)$.
Let
$u\in \mathbf{L^\infty}\left((0,+\infty)^2; \R\right)$ be the 
entropy   admissible  weak solution  of the initial-boundary value problem \begin{equation}
    \label{eq:ibvp}
    \begin{cases} 
          \pt u + \px f(u) = 0, & x >0, \, t > 0,
          \\
          u(0,x) = \overline{u}(x), & x>0,\\
          u(t,0)= k(x), & t>0.
        \end{cases}
    \end{equation}
    Then, we have
    \begin{equation}
    \label{eq:mainest}
    \tv_{(0, +\infty)} f( u(\cdot\,, 0))
    \leq
    2\tv_{(0, +\infty)} (f\circ k)
    +\|f'(\overline u)\|_\infty   \tv_{(0, +\infty)} ( \overline{u}) 
    +\os(f)\,,
    \end{equation}
    with
    \begin{equation}
    \label{eq:defosc}
       \os(f)\doteq\sup\Big\{\big| f\circ k(t) - f\circ \overline u(x)\big|,~~ t>0,\  x>0\Big\}.
    \end{equation}
\end{theorem}
\begin{proof}
  We divide the proof in several steps.  \smallskip

  \noindent
  {\bf Step 1. Properties of the solution $u$ at the boundary.}

  Since $f\circ k\in \mathbf{BV}$ and $k (t)\le \theta$ for a.e. $t > 0$, it follows that
  $k$ admits
  at most countably many discontinuities and one-sided limits~$k(t^\pm)$ at any time $t>0$, because 
  $f\circ k$ enjoys these properties and $(f_-)^{-1}$ is continuous
  ($f_-$~defined as in Section~\ref{sec:bvp}).
  We may assume that $k$ is right continuous i.e. that $k(t)=k(t^+)$ for all $t>0$.

  Let $t>0$  such that $u(t,0)\leq \theta$. Then the following holds:  
  \begin{itemize}
  \item[(i)] 
  Due to the  boundary condition in~\eqref{eq:ibvp} (e.g.~see~\cite[Section~2.1]{lf}),  it holds $u(t,0)=k(t)$ if   $u(\cdot, 0)$ is continuous at $t$, otherwise  $u(\cdot, 0)$
  admits one-sided limits, and we have
  $u(t,0)=u(t^-,0)=k(t^-)$, $u(t^+,0)=k(t)$.

  \item[(ii)] If $u(t,0)<\theta$,  tracing backward generalized characteristics 
  starting at  points $(t,x)$, with $x>0$ arbitrary close to $0$, 
  we  deduce that 
  \begin{equation}
      \label{eq:leftext-tau'}
      \exists~\tau'<t\qquad{s.t.}\qquad u(s,0)< \theta\quad\forall~s\in(\tau', t).
  \end{equation}

  \item[(iii)] If $u(t,0)=\theta$, then, by item (i), it holds 
    $u(t^-,0)=\theta=k(t)$. Moreover, we deduce that
    \begin{equation}
      \label{eq:leftext-tau'-2}
      \exists~\tau'<t\qquad{s.t.}\qquad u(s,0)\leq  \theta\quad\forall~s\in(\tau', t).
    \end{equation}
    If it were not true, there would exist  a sequence $\{t_n\}_n$, with 
    $t_n \uparrow t$, such
    that $ u(t_n, 0)\downarrow \theta $   
    and so  $f'(u(t_n,0))\uparrow 0$. Considering the maximal backward characteristics 
    starting at the points $(t_n,0)$, this would imply that they must cross each other in the domain $\{x>0\}$ 
    for $n$ sufficiently large, which is not possible.
  \end{itemize}


  \noindent 
  {\bf Step 2. Definition of the generalized characteristics issuing from the boundary. }

  We will say that a map $\xi\colon [\tau_1, \tau_2] \to \R$,
  $\tau_2\in \R^+\cup{+\infty}$,
  is a generalized characteristic starting at the point $(\tau_1, 0)$ if:
  \vspace{-2pt}
  \begin{itemize}
    \item[-]  $\xi(\tau_1)=0$;
    \vspace{-2pt}
    \item[-]  either $\tau_2 =+\infty$ or $\xi(\tau_2)=0$;
    \vspace{-2pt}
    \item[-] $\xi(t)>0$ for all $t\in (\tau_1,\tau_2)$ and the restriction 
      of $\xi$ to $(\tau_1,\tau_2)$ is a generalized characteristic 
      in the classical sense (e.g. see~\cite{d,daf-book}).
  \end{itemize}
  For all $t>0$ such that 
  \begin{equation}
  \label{eq:condexistforwchar}
      u(t,0)<\theta\quad  \text{  or  } \quad k(t)<k(t^-)=u(t,0)=\theta,\quad \text{ or   }\quad 
  k(t)<\theta<u(t,0)\leq \pi_+(k(t)),
  \end{equation} where $\pi_+(u)\doteq (f_+)^{-1} (f(u))$
  ($f_+$ defined as in Section~\ref{sec:bvp}) 
  we denote   $s\to \xi_{t}(s)$   the {\it maximal forward generalized characteristic} starting at $(t,0)$, i.e. the forward generalized characteristic starting at $(t,0)$ with slope 
  \begin{equation}
    \label{eq:slopeforwchar}
    \lambda_t \doteq
    \begin{cases}
        \ \ f'(u(t,0)),\qquad
        &\text{if}\quad \ u(t,0) <\theta, 
        \\
        \noalign{\bigskip}
        \dfrac{f(u(t,0))-f(k(t))}{u(t,0)-k(t)},
        \qquad
        & \!\!\!
        \begin{array}{l}
          \text{if}\quad \ k(t)<\theta<u(t,0)\leq \pi_+(k(t)),
          \\
          \text{or}\quad 
          k(t)<k(t^-)=u(t,0)\leq \theta.
        \end{array}
          \end{cases}
  \end{equation}

  We denote $\text{Dom}(\xi_t)=[t, \hat t\,]$, $\hat t\in \R^+\cup{+\infty}$, the domain of existence of
  $\xi_t$.
  Observe that $\lambda_t\geq 0$ by definition~\eqref{eq:slopeforwchar}.
  Moreover, if $\hat t\in \R^+$, then we have
  $\xi_t(\,\hat t\,)=0$. 
 
  \smallskip

  \noindent
  {\bf Step 3. Definition of   $\tau^*$ and computation of total variation of $f(u(t,0))$ in $(\tau^*, +\infty)$.} 

  Let 
  \begin{equation}
      \tau^*\doteq\inf\Big\{t>0~\big|~~\eqref{eq:condexistforwchar} \ \  \text{holds and}\quad \text{Dom}(\xi_t)=[t, +\infty) \Big\}.
      \label{eq:tau*def}
  \end{equation}
  Note that if \eqref{eq:condexistforwchar} does not hold for any $t>0$, 
  then there is no   forward generalized characteristics starting at $(t,0)$,
  and we set in this case $\tau^*=+\infty$. 
  Moreover, it may happen that for all $t>0$ for which \eqref{eq:condexistforwchar} holds, 
  it holds that $\text{Dom}(\xi_t)=[t, \hat t]$ with $\hat t\in \R^+$. 
  Also in this case by definition $\tau^*=+\infty$.

Let $\tau^*<+\infty$. 
Observe that 
  \begin{equation}
  \label{eq:trace u-above-tau*}
      u(t,0)\leq \theta\qquad\forall~t>\tau^*.
  \end{equation}
  In fact, if $u(t,0)> \theta$ 
  at some $t>\tau^*$, we could trace the maximal backward characteristics $\vartheta_t$ starting at $(t,0)$,
  which would meet at some point $x>0$
  a forward generalized characteristic~$\xi_{t'}$ starting at some $(t',0)$,
  $t'\in (\tau^*,t)$.
  But a backward characteristic
  cannot cross a forward generalized characteristic.
  We deduce from~\eqref{eq:trace u-above-tau*}, recalling  property (i)  at {\bf Step 1}, that 
  \begin{equation}
  \label{eq:tvf-tau^*}
      \tv_{(\tau^*, +\infty)} f( u(\cdot\,, 0))
      =
      \tv_{(\tau^*, +\infty)} (f\circ k).
  \end{equation}

  We observe that if $\tau^*<+\infty$, actually it is the minimum in \eqref{eq:tau*def} since the set of generalized characteristic is closed with respect to uniform convergence. \\
  \noindent
  {\bf Step 4. The variation of  $f(u(t, 0))$ at $t=\tau^*<+\infty$.} 

  We show first of all  that
  \begin{equation}\label{utau8}
      u(\tau^*,0)>\theta.
  \end{equation}
  If it were not true,  either 
  $u(\tau^*,0)<\theta$ or $u(\tau^*,0)=\theta$, and we are going to show that both these possibilities cannot verify. 
  \begin{enumerate}
      \item Assume $u(\tau^*,0)<\theta$.
      Then, by \eqref{eq:leftext-tau'}, there exists $\tau'<\tau^*$ 
  with $u(t,0)< \theta$ for $t\in(\tau', \tau^*)$. 
  Then, consider the maximal forward generalized characteristic
  $\xi_t$ starting at some point
  $(t,0)$, $t\in (\tau', \tau^*)$.
  Observe that, if $\text{Dom}(\xi_t)=[t, \hat t\,]$, $\hat t\leq \tau^*$, 
  then we would have $\xi_t(\hat t\,)= 0$, \, $\xi_t'(\hat t\,)\leq 0$.
  Hence,   the Lax condition  implies that  $f'(u(\hat t,0))\leq \xi_t'(\hat t\,)\leq 0$
  which in turns gives  $u(\hat t,0)\geq \theta$,
  in contradiction with~\eqref{eq:leftext-tau'} since $\hat t \in (\tau', \tau^*)$.
  If $\text{Dom}(\xi_t)=[t, \hat t\,]$, $\hat t> \tau^*$, then 
  $\hat t<+\infty$ by minimality of $\tau^*$, recalling definition~\eqref{eq:tau*def}. 
  Therefore,  there exists 
  $\widetilde{\tau}>\tau^*$ such that
  $\xi_t(\,\widetilde{\tau }\,)=\xi_{\tau^*}(\,\widetilde{\tau }\,)$ and then  $\xi_t(s)$ coincides with $\xi_{\tau^*}(s)$
  for all $s\geq \widetilde{\tau }$. So    again it will give $\hat t = +\infty$ in contradiction with definition~\eqref{eq:tau*def}.
    \item Assume  $u(\tau^*,0)=\theta$.
      By property \eqref{eq:leftext-tau'-2},  we may  define
  \begin{equation}
  \label{eq:taubar1-def}
      {\tau}_1\doteq\inf\Big\{
      \tau'\in (0,\tau^*\,)
      ~\big|~ 
      u(t,0)\leq \theta\quad\forall~t\in(\tau', \tau^*)
      \Big\}.
  \end{equation}
  By definition~\eqref{eq:taubar1-def} it follows that 
  \begin{eqnarray}
  \label{eq:tau1-ineq1}
      u({\tau}_1,0)>\theta.
  \end{eqnarray}
  In fact, if 
  $u({\tau}_1,0)\leq \theta$, by the same arguments above we would deduce that there exists ${\tau}'<{\tau}_1$
  such that $u(t,0)\leq \theta$ for all $t\in ({\tau}', {\tau}_1)$, contradicting the definition~\eqref{eq:taubar1-def}. 
  Moreover, observe that, relying again on 
  properties (i)-(ii) at   {\bf Step 1}, and because of definition~\eqref{eq:taubar1-def}, we have $u({\tau}_1^+,0)=k({\tau}_1)$,
  and
  \begin{equation}
      \label{eq:tau1-ineq2}
      k({\tau}_1)
      < \theta,\qquad\quad
      \pi_+(k({\tau}_1))\geq u({\tau_1},0).
  \end{equation}
  In fact, if either $k({\tau}_1)= \theta$, or $\pi_+(k({\tau}_1))< u({\tau}_1,0)$,
  then the initial boundary Riemann problem on $\{t\geq {\tau}_1, \, x\geq 0\}$, 
  with initial datum $u({\tau}_1,0)$, and boundary datum $k({\tau}_1)$,
  would be solved by the constant function
  $u({\tau}_1,0)>\theta$, which is incompatible with the fact that 
  $u({\tau}_1^+,0)\leq \theta$
  by definition~\eqref{eq:taubar1-def}.
  Therefore, due to  \eqref{eq:taubar1-def},
  \eqref{eq:tau1-ineq2}, condition~\eqref{eq:condexistforwchar}
  is verified at time $t={\tau}_1$, and we may consider the maximal forward generalized characteristic $\xi_{{\tau}_1}$ starting at $({\tau}_1,0)$, 
  ${\tau}_1<\tau^*$.
  By  the same arguments of the previous case we deduce that
  $\text{Dom}(\xi_{{\tau}_1})=[{\tau}_1, +\infty)$,
  which is contrast with the minimality of $\tau^*$ in Definition~\eqref{eq:tau*def}.
  \end{enumerate}

  Observe then that due to \eqref{utau8} and to the Definition of $\tau^*$, it holds that 
  $f(u((\tau^*)^+,0))=f(k((\tau^*)^+))$,
  and that
  $f(u(\tau^*,0))\in \big[f(\overline{u}((x^*)^-), 
  f(\overline{u}((x^*)^+)]$ where $x^*$ is  the position at $t=0$ of the maximal backward characteristic  starting at $(\tau^*, 0)$. 
  We then have that 
  \begin{equation}
  \label{osc-est}
      \big|f(u((\tau^*)^+,0))-f(u(\tau^*,0))\big|
      \leq \os(f),
  \end{equation}
  with $\os(f)$ defined as in~\eqref{eq:defosc}. 

  \smallskip

  \noindent
  {\bf Step 5. Structure of the trace of the  solution $u(\cdot, 0)$ in $(0,\tau^*)$.} 

  We claim that there exist at most countably many intervals 
  $(\tau_{n,1},\tau_{n,2})\subset [0,\tau^*]$,
  $n\in \N$, with the property
  \begin{equation}
  \label{eq:prop-decomp-1}
      u(\tau_{n,i},0)> \theta,\quad i=1,2,\qquad\qquad 
      u(t,0)\leq  \theta,\qquad\forall~t\in 
      (\tau_{n,1},\tau_{n,2}),
  \end{equation}
  such that 
  \begin{equation}
  \label{eq:prop-decomp-1b}
      \Big\{t\in (0,\tau^*) ~\big|~ u(t,0)\leq  \theta
      \Big\} = \bigcup_n~ (\tau_{n,1},\tau_{n,2})\,;
  \end{equation}
  see \Cref{fig:structure-tau}.

  \begin{figure} 
    \centering
    \begin{tikzpicture}[line cap=round,line join=round,x=1.2cm,y=1.2cm]
  

      \draw[->, line width=1.1pt] (1, .5) -- (1, 7.) node[below right]{$t$};
      \draw[->, line width=1.1pt] (1, .5) -- (8.5, .5) node[above]{$x$};


      \node[inner sep=0, anchor=east] at (0.9, 6.4) {$\tau^*$};

      \draw[line width=1.1pt, color=blue] (1, 0.8) node[left] {\color{black} $\tau_{1,1}$} to [out=40, in=270] (1.3, 1.15) 
      to [out=90, in=-40] (1, 1.5) node[left] {\color{black} $\tau_{1,2}$}; 
      \draw[line width=1.1pt, color=blue] (1, 2.) node[left] {\color{black} $\tau_{2,1}$} to [out=20, in=270] (1.6, 2.5) 
      to [out=90, in=-20] (1, 3.) node[left] {\color{black} $\tau_{2,2}$}; 
      \draw[line width=1.1pt, color=blue] (1, 4.) node[left] {\color{black} $\tau_{3,1}$} to [out=35, in=270] (1.4, 4.5) 
      to [out=90, in=-35] (1, 5) node[left] {\color{black} $\tau_{3,2}$}; 

      \draw[line width=.8pt] (1, 0.8) -- (1.6, 0.5);
      \draw[line width=.8pt] (1, 1.5) -- (3.2, 0.5);
      \draw[line width=.8pt] (1, 2.) -- (3.6, 0.5);
      \draw[line width=.8pt] (1, 3.) -- (7.5, 0.5);
      \draw[line width=.8pt] (1, 4.) -- (7.8, 0.5);
      \draw[line width=.8pt] (1, 5.) -- (8.2, 0.5);

      \draw[line width=1.pt] (1, 6.5) -- (2.2, 6.75) node[below] {\color{black}$\xi_{t}$};
      \draw[line width=1.pt, dashed] (2.2, 6.75) -- (3.4, 7.);

    \end{tikzpicture}
    \caption{A possible configuration, in the $(t, x)$ plane, 
      related to the intervals $(\tau_{n, 1}, \tau_{n,2}) \subseteq (0, \tau^*)$, described
      in~\eqref{eq:prop-decomp-1}-\eqref{eq:prop-decomp-1b}, with $n=1, 2, 3$. 
      The maximal forward generalized characteristics $\xi_t$ (see Step~2), for $t > \tau^*$, is defined on $[t, +\infty)$.
      Note that the generalized backward characteristics starting form points $(t, 0)$, with time $t < \tau^*$ not in
      $(\tau_{n, 1}, \tau_{n,2})$, enter the domain and reach, at time $0$, a point $x > 0$.
    }
        \label{fig:structure-tau}
  \end{figure}
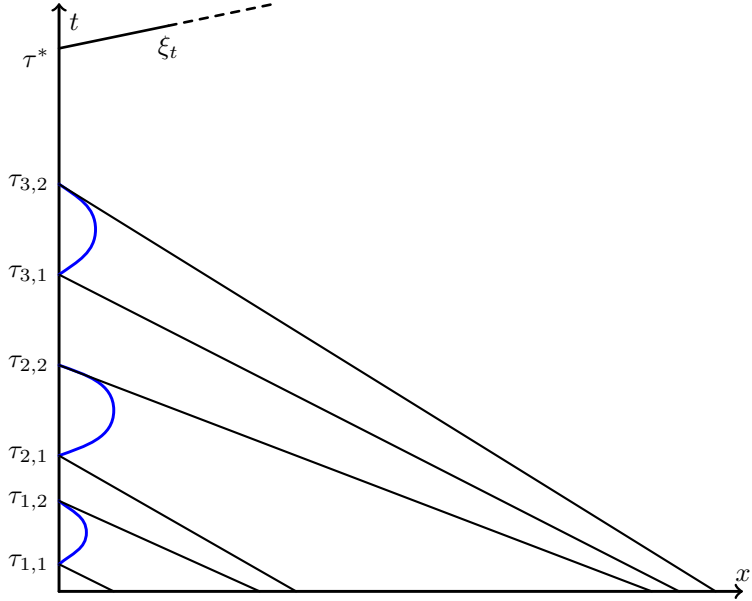

  We start fixing any time
  $\tau\in (0,\tau^*)$ where 
  $u(\tau,0)\leq \theta$.  
  We claim that there exist $\tau_1<\tau_2\leq \tau^*$ such that  $\tau\in (\tau_1, \tau_2)$ and \begin{equation}
  \label{eq:prop-decomp-4}
      u(\tau_{i},0)> \theta,\quad i=1,2,\qquad\qquad 
      u(t,0)\leq  \theta,\qquad\forall~t\in 
      (\tau_{1},\tau_{2}).
  \end{equation}
  We observe that since $u(\tau, 0)\leq \theta$, by   properties \eqref{eq:leftext-tau'}, \eqref{eq:leftext-tau'-2} there exists $\tau'<\tau$ such that  $u(t,0)\leq \theta$ for $t\in(\tau', \tau)$.
  Let us define 
  \[
      {\tau}_1\doteq\inf\Big\{
      \tau'\in (0,\tau\,)
      ~\big|~ 
      u(t,0)\leq \theta\quad\forall~t\in(\tau', \tau)
      \Big\}\,.
  \]
  Then, $\pi_+(k({\tau}_1))\geq u(\tau_1, 0)>\theta\geq k(\tau_1)$.

  Next, consider the maximal forward generalized characteristic $\xi_{{\tau}_1}$ starting at $({\tau}_1,0)$. 
  Since ${\tau}_1<\tau<\tau^*$, it follows from the minimality of $\tau^*$ in ~\eqref{eq:tau*def} that 
  $\text{Dom}(\xi_{{\tau}_1})=[{\tau}_1, {\tau}_2]$, for some ${\tau}_2\leq \tau^2$, and $\xi_{{\tau}_1}({\tau}_1)=\xi_{{\tau}_1}({\tau}_2)=0$. Arguing as in \eqref{eq:trace u-above-tau*} we deduce that 
  \[ 
      u(t,0)\leq\theta
      \qquad\forall~t\in 
      ({\tau}_1,{\tau}_2).
  \]
  Observe that by Lax condition we have $f'(u({\tau}_2,0))\leq \xi'_{{\tau}_1}({\tau}_2)\leq 0$,
  which implies $u({\tau}_2,0)\geq\theta$.
  We will show that necessarily $u(\tau_2, 0)>\theta$. 
  Indeed,  $u(\cdot, 0)$ admits a discontinuity at $t={\tau}_2$ connecting the left state $u({\tau}_2^-, 0)$ with the right state $u({\tau}_2, 0)$, and having slope $\xi'_{{\tau}_1}({\tau}_2)$.
  If $u({\tau}_2,0)=\theta$, then it follows that $\xi'_{{\tau}_1}({\tau}_2)=0$,
  but there is no discontinuity with zero slope and right state $\theta$. Therefore, it must be
    $ u({\tau}_2,0)>\theta$ and then necessarily $\tau_2\leq\tau^*$. 
    This concludes the proof of the claim \eqref{eq:prop-decomp-4}.

  Finally observe that the interval $(\tau_1, \tau_2)$   is associated to 
  the discontinuity curve $(t,\xi_{{\tau}_1}(t))$ starting at $({\tau}_1, 0)$ and ending at 
  $({\tau}_2, 0)$.
  Since any solution admits at most countably many curves of discontinuity, repeating the 
  construction of the interval $({\tau}_1, {\tau}_2)$
  starting from any time
  $\tau\in (0,\tau^*)$, we deduce the existence
  of (at most) countably many intervals 
  $(\tau_{n,1},\tau_{n,2})\subset [0,\tau^*]$,
  $n\in\N$,
  enjoying the properties~\eqref{eq:prop-decomp-1}, \eqref{eq:prop-decomp-1b}.
  \smallskip

  \noindent
  {\bf Step 6. The computation of the total variation of $f(u(t,0))$ on a single interval $[\tau_{n,1}, \tau_{n,2}]$.} 

  Let $\tau_1<\tau_2\leq\tau^*$ be such that
  \eqref{eq:prop-decomp-4} holds. 

  We show that 
  \begin{equation}
  \label{eq:esttv-1}
  \begin{aligned}
      &\big|f(u(\tau_{2},0))-f(  u(\tau_{2}^-,0))\big|
      +\big|f(u(\tau_{1}^+,0))-f(u(\tau_{1},0))\big|
      \\
      \noalign{\smallskip}
      &\hspace{1.7in} \leq \big|f( u(\tau_{2},0))-f(u(\tau_{1},0))\big|
      + \tv_{(\tau_{1}, \tau_{2})} (f\circ k).
  \end{aligned}
  \end{equation}
  Observe that by property (i) at  {\bf Step 1}, and because of~\eqref{eq:prop-decomp-4}, 
  we have $f(u(\tau_{2}^-,0))=f(k(\tau_{2}^-))$,
  $f(u(\tau_{1}^+,0))=f(k(\tau_{1}))$.

  Notice that, by the analysis at   {\bf Step 5}, since there is a shock discontinuity at $(0,\tau_1)$ with nonnegative slope connecting the left state $u(\tau_1^+,0)$ with the right state $u(\tau_1,0)$, it follows that 
  \begin{equation}
  \label{fluxineq_1}
      f(u(\tau_1^+,0)) \leq f(u(\tau_1,0)).
  \end{equation}
  Similarly, we have
  \begin{equation}
  \label{fluxineq_2}
      f(u(\tau_2^-,0)) \geq f(u(\tau_2,0)),
  \end{equation}
  since at $(0,\tau_2)$
  there is a shock discontinuity with nonnegative slope connecting the left state $u(\tau_2^-,0)$ with the right state $u(\tau_2,0)$.
  Therefore, we have 
  \begin{equation*}
  \label{eq:esttv-2}
  \begin{aligned}
      &\big|f(u(\tau_{2},0))-f(  u(\tau_{2}^-,0))\big|
      +\big|f(u(\tau_{1}^+,0))-f(u(\tau_{1},0))\big|
      \\
      \noalign{\smallskip}
      &\hspace{1.5in} = f( u(\tau_{2}^-,0))-f(u(\tau_{1}^+,0))
      + f( u(\tau_{1},0))-f(u(\tau_{2},0))
      \\
      \noalign{\smallskip}
      &\hspace{1.5in} \leq \big|f( u(\tau_{2},0))-f(u(\tau_{1},0))\big|
      + \big|f( u(\tau_{2}^-,0))-f(u(\tau_{1}^+,0))\big|.
  \end{aligned}
  \end{equation*}
  This implies that ~\eqref{eq:esttv-1} holds.

  Now, we  compute the total variation. 
  Consider a $(N+1)$-tupla of times $t_0\leq \tau_1<t_1<\dots< t_{N-1}<\tau_2\leq t_N$, with 
  \begin{equation}
  \label{eq:TVest-1}
      u(t_0,0)>\theta,\qquad\qquad 
      u(t_N,0)>\theta.
  \end{equation}
  So, recalling \eqref{eq:esttv-1}, we compute 
  \begin{eqnarray}
  \nonumber 
  &&
      \sum_{k=0}^{N-1}
      \big|f(u(t_{k+1},0))-f(u(t_k),0)\big|\\ \nonumber 
  &\leq &   
      \big|f( u(t_N,0))-f(u(\tau_{2},0))\big|+\big|f( u(\tau_{1},0))-f(u(t_0,0))\\\nonumber 
      &&
      +
      \big|f( u(\tau_{2},0))-f(u(\tau_{2}^-,0))\big|
    +\tv_{(\tau_{1}, \tau_{2})} (f\circ k)+
      \big|f( u(\tau_{1}^+,0))-f(u(\tau_{1},0))\big| 
      \\\nonumber 
      &\leq  & \big|f( u(t_N,0))-f(u(\tau_{2},0))\big|+\big|f( u(\tau_{1},0))-f(u(t_0,0)) \big|
      \\ 
      &&
      +
      \big|f( u(\tau_{2},0))-f(u(\tau_{1},0))\big|
    +2\tv_{(\tau_{1}, \tau_{2})} (f\circ k). \label{eq:esttv-3}
  \end{eqnarray}
  \smallskip

  \noindent
  {\bf Step 7. Computation of the total variation of $f(u(t,0))$ on the entire interval $(0, \tau^*)$.}
  Let $\big\{\tau_{n,1}, \tau_{n,2}\big\}_n$,
  be as in~\eqref{eq:prop-decomp-1}, \eqref{eq:prop-decomp-1b}.

  We fix a  $(N+1)$-tupla of times $0\leq t_0< t_1<t_2<\dots< t_{N}\leq \tau^*$.

  Since $t_j$ are  finite, there exists a finite number $\bar n\geq 0$ of intervals $(\tau_{n,1}, \tau_{n,2})$ which contains at least one of the $t_j$. Note that $\bar n$ could also be $0$ (in particular this will always happen when \eqref{eq:condexistforwchar} for any $t>0$). 

  From now on we assume $\bar n>0$, the case $\bar n=0$ can be treated in an analogous (and simpler) way. 
  Up to relabeling the indexes we may assume that this holds for $n=1,\dots, \bar n$, with $\tau_{n,1}<\tau_{n+1, 2}$ for all $n=1, \dots, \bar n-1$.

  Therefore, for any $n\in \{1, \dots, \bar n\}$, we have that  the intersection 
  $ (\tau_{n,1}, \tau_{n,2})\cap\{t_0, \dots, t_N\}$ is not empty and contains $m_n\geq 1$ elements, that up to relabeling we denote as the ordered sequence  \[\{t_{n,1}, t_{n, 2}, \dots, t_{n, m_{n}}\}=(\tau_{n,1}, \tau_{n,2})\cap\{t_0, \dots, t_N\}.\]
  In particular, we have that
  \[u(t_{n, j}, 0)\leq\theta,\qquad  j=1,\dots, m_n,\ n=1,\dots, \bar n.\]

  Moreover, for every $n=1, \dots,\bar n-1$
  there exists a finite number $p_n\geq 0$ of $t_j$ which are contained in $[\tau_{n, 2}, \tau_{n+1, 1}]$. 
  Let us denote them as the ordered sequence 
  \begin{equation*}
    \{t_{n,n+1,1},  \dots, t_{n, n+1,p_{n}}\}=[\tau_{n,2}, \tau_{n+1,1}]\cap\{t_0, \dots, t_N\},\qquad n=1, \dots, \bar n-1.    
  \end{equation*}
  Finally, there are at most $p_0\geq 0$ times $t_j$ such that $t_{j}\leq \tau_{1,1}$. We label them as $\{t_{0,1, 1},\dots, t_{0,1,p_0}\}$. Analogously there are  at most $p_{\bar n}\geq 0$ times $t_j$ such that $t_{j}\geq \tau_{\bar n,2}$, and we label them as $\{t_{\bar n,\bar n+1, 1},\dots, t_{\bar n,\bar n+1,p_{\bar n}}\}$ (maintaining the order). 

  In particular, we have that
  \[u(t_{n,n+1, j}, 0)>\theta, \qquad j=1,\dots, p_n, \ n=0, \dots, \bar n+1.\]

  In conclusion, we have relabeled the $(N+1)-$tupla of times as follows: 
  \[\{t_0, t_1, \dots, t_k, \dots t_N\}= 
  \bigcup_{n=0}^{\bar n+1}\bigcup_{j=1}^{p_n}\{t_{n,n+1, j}\}\cup \bigcup_{n=1}^{\overline n} \bigcup_{j=1}^{m_n}\{t_{n, j}\}.
  \]
  Then, using this relabelling and applying estimate~\eqref{eq:esttv-3} at the previous  {\bf Step 6}, we find
  \begin{eqnarray}
  \label{eq:esttv-4}
  && \sum_{k=0}^{N-1} |f(u(t_{k+1},0))-f(u(t_k,0))|
  \\\nonumber  &\leq & \sum_{n=0}^{\bar n+1}\sum_{j=1}^{p_n-1}|f(u(t_{n,n+1,j+1},0))-f(u(t_{n,n+1, ,j},0))| 
      \\ \nonumber 
      && + \sum_{n=0}^{\overline n-1}
      |f(u(\tau_{n+1,1},0))-f(u(t_{n,n+1, p_n},0))|
      + \sum_{n=1}^{\overline n+1}
      |f(u(\tau_{n,2},0))-f(u(t_{n,n+1, 1},0))|\\
      \nonumber && + \sum_{n=1}^{\overline n}
      |f(u(\tau_{n,1},0))-f(u(\tau_{n,2},0))|
      + 2 \sum_{n=1}^{\overline n} 
      \tv_{(\tau_{n,1}, \tau_{n,2})} (f\circ k).
      \end{eqnarray}
  Let  $\vartheta_{n,n+1, j}(t)$
  be the maximal backward characteristic starting at $(t_{n,n+1,j}, 0)$, for $j=1, \dots, p_n$ and  $n=0, \dots, \bar n+1$,   and denote   $x_{n,n+1,j}=\vartheta_{n,n+1,j}(0)$. Then the  initial datum $\overline u$ will be  either   continuous at $x_{n, n+1, j}$, or it will 
  admit at $x_{n,n+1, j}$ a downward jump
  $\overline u(x_{n,n+1,j}^-)>\overline u(x_{n,n+1,j}^+)$ generating a rarefaction.  
In particular, for every $j=1,\dots, p_n-1$, we have 
 \begin{equation}\label{eq:flux-from-indatum-1}  |f(u(t_{n,n+1, j}, 0))-f(u(t_{n,n+1,j+1},0))|\leq \|f'(\overline u)\|_\infty \tv_{[x_{n,n+1,j}, x_{n,n+1,j+1}]}(\overline u),\end{equation} where $\tv_{[a,b]}(\overline u)$ is defined in \eqref{tvchiusa}. 
  
  Similarly, we may consider the 
  maximal backward characteristic $\zeta_{n,i}$ starting at $(\tau_{n,i}, 0)$, $i=1,2$,
  set $z_{n,i}=\zeta_{n,i}(0)$, $i=1,2$,
  and conclude that
  \begin{eqnarray}
 \nonumber 
      &|f(u(\tau_{n+1,1},0))-f(u(t_{n,n+1, 
      p_n},0))| &\leq \|f'(\overline u)\|_\infty\tv_{[z_{n+1,1}, x_{n,n+1,p_n}]}(\overline u),\\\nonumber 
     & |f(u(\tau_{n,2},0))-f(u(t_{n,n+1, 1},0))|&\leq 
     \|f'(\overline u)\|_\infty\tv_{[z_{n,2}, x_{n,n+1,1}]}(\overline u),\\
      &|f(u(\tau_{n,1},0))-f(u(\tau_{n,2 },0))|&
      \leq \|f'(\overline u)\|_\infty \tv_{[z_{n,2}, z_{n,1}]}(\overline u). \label{eq:flux-from-indatum-2}
  \end{eqnarray}
If $x_{n,n+1,j}\neq x_{n,n+1,j+1}\neq z_{n,i}$ for all $i,j$,  relying on~\eqref{eq:flux-from-indatum-1}, \eqref{eq:flux-from-indatum-2},
  we derive from~\eqref{eq:esttv-4} that
  \begin{equation}
  \label{eq:esttv-5}
      \sum_{k=0}^{N-1}
      \big|f(u(t_{k+1},0))-f(u(t_k,0))\big|
      \leq \|f'(\overline u)\|_\infty\tv_{(0, +\infty)} (\overline{u}) +2\tv_{(0, \tau^*)} (f\circ k).
  \end{equation}
 Note that it could happen that $x_{n,n+1,j-1}<x_{n,n+1,j}=x_{n,n+1,j+k}<x_{n,n+1,j+k+1}$ for some $j, k>0$. In this case, instead of relying on the estimates in~\eqref{eq:flux-from-indatum-1}, we 
 estimate all together 
 \begin{equation}
     \label{eq:flux-from-indatum-3} \sum_{i=j-1}^{j+k}
\big|f(u(t_{n,n+1, i}, 0))-f(u(t_{n,n+1,i+1},0))\big|\leq \big\|f'(\overline u)\big\|_\infty 
 \tv_{[x_{n,n+1,j-1},\, x_{n,n+1,j+k+1}]}(\overline u). 
 \end{equation}
Similarly, we treat the cases in which $z_{n,2}<x_{n,n+1,1}=x_{n,n+1,1+k}<x_{n,n+1,2+k}$, for some  $k$, getting
\[\big|f(u(\tau_{n,2}, 0))-f(u(t_{n,n+1,1},0))\big|+\sum_{i=1}^{2+k}
\big|f(u(t_{n,n+1, i}, 0))-f(u(t_{n,n+1,i+1},0))\big| \]\[\leq   \big\|f'(\overline u)\big\|_\infty 
 \tv_{[z_{n,2},\, x_{n,n+1,2+k}]}(\overline u)
 \]
or $x_{n,n+1,j-1}<x_{n,n+1,j}=x_{n,n+1,j+k}<x_{n,n+1,j+k+1}<z_{n+1,1}$, for some $j$, or 
$z_{n,2}=z_{n,1}$ or $z_{n,2}=x_{n,n+1,1}$ or $z_{n+1,1}=x_{n, n+1, p_n}$ for some $n$.
Thus, in any case,
combining~\eqref{eq:flux-from-indatum-1}, \eqref{eq:flux-from-indatum-2}with~\eqref{eq:flux-from-indatum-3}, we derive the bound~\eqref{eq:esttv-4}.
  \smallskip

  \noindent
  {\bf Step 8. Conclusion.}

  Given any $(N+1)$-tupla of times $t_0< t_1<t_2<\dots< t_{N}$,
  combining~\eqref{eq:tvf-tau^*} with \eqref{eq:esttv-5}
  we deduce
  \begin{align*}
      \sum_{k=0}^{N-1}
      \big|f(u(t_{k+1},0))-f(u(t_k,0))\big|
      \leq \|f'(\overline u)\|_\infty \tv_{(0, +\infty)} (\overline{u}) +2\tv_{(0, +\infty)} (f\circ k)\\+\big|f(u((\tau^*)^+,0))-f(u(\tau^*,0))\big|\,.
  \end{align*}
  From this, recalling estimates \eqref{eq:tvf-tau^*} and  \eqref{osc-est},  we get  the proof of the theorem. 
\end{proof}

\begin{remark}\upshape
  \label{sharphyp-bv-fluxtrace}
  We observe that the assumptions in \Cref{bvfluxbdrytrace} are sharp.

  Observe that if   $f\circ k\not\in \mathbf{BV}((0, +\infty); \R)$  then the result is not true. 
  
  Let 
  $k \in \mathbf{L^\infty}\left((0,+\infty)\,; (-\infty,\theta]\right)$, with
  $f\circ k\not\in \mathbf{BV} ((0,+\infty); \R)$,
  and $\overline u \in \mathbf{BV}\left((0,+\infty)\,; (-\infty,\theta]\right)$.
  Then the 
  entropy   admissible  weak solution  $u$ of the initial-boundary value problem~\eqref{eq:ibvp}
  will take values in $(-\infty,\theta]$
  as well, and we have
  $u(\cdot , 0)=k$; see~\cite{lf}.
  Therefore, in this case we have $f\circ u(\cdot, 0)=f\circ k\not\in \mathbf{BV} ((0,+\infty); \R)$.

  Next, we show with a counterexample that if $ \overline u\not\in \mathbf{BV}((0, +\infty); \R)$  
  then the result is not true. Consider the  flux $f(u)=-\frac{u^2}{2}$, the boundary datum $k(t)\equiv 0$ and 
  the  initial datum
  \begin{equation}
    \overline{u}(x)\doteq
    \begin{cases}
      \frac{1}{2}, & x\in \left(\frac{4}{3}
      \frac{1}{2^n}, \frac{1}{2^{n-1}}\right), 
      n\geq 1,\\
      1, &\text{otherwise}. 
    \end{cases}
  \end{equation}

  Clearly we have  $ \overline u\not\in \mathbf{BV}(0, +\infty)$. By direct computation one can check that the trace of the solution satisfies:
  \[
    u(t, 0)=
    \begin{cases}
      1, & t>\frac{4}{3},\\
      \frac{2}{3 2^{n-1} t}, & t\in \left(\frac{4}{3}
      \frac{1}{2^n}, \frac{4}{3}\frac{1}{2^{n-1}}\right), 
      n\geq 1;
    \end{cases}
  \]
  see \Cref{fig:ex-appendice}.
  In particular observe that
  \[u\left(\left(\frac{4}{3}
  \frac{1}{2^n}\right)^+, 0\right)=1\qquad u\left(\left(\frac{4}{3}
  \frac{1}{2^{n-1}}\right)^-, 0\right)=\frac{1}{2}\]
  which implies immediately that 
  \[
    \tv_{(0,4/3)} u(t,0)\geq\sum_{n=0}^{+\infty} \left(1-\frac{1}{2}\right)=+\infty 
  \] 
  and similarly 
  \[
    \tv_{(0,4/3)} f(u(t,0))\geq\sum_{n=0}^{+\infty} \left(\frac{1}{2}-\frac{1}{8}\right)=+\infty. 
  \] 
  Therefore, $f(u(\cdot, 0))\not\in \mathbf{BV}((0,4/3); \R)$.

  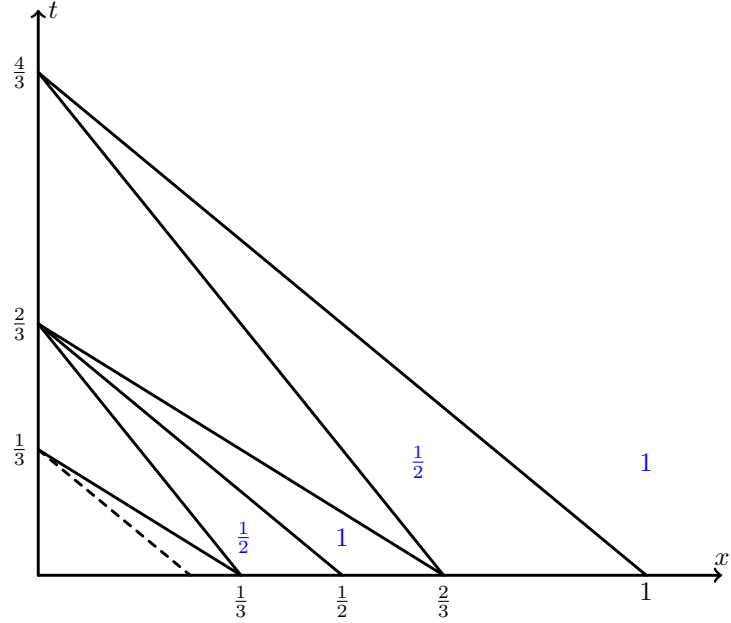
\begin{figure}
    \centering
    \begin{tikzpicture}[line cap=round,line join=round,x=1.cm,y=1.cm]
  
      \draw[->, line width=1.1pt] (2., .5) -- (2., 8.) node[right]{$t$};
      \draw[->, line width=1.1pt] (2., .5) -- (11., .5) node[above]{$x$};

      \node[inner sep=0, anchor=north] at (10, 0.4) {$1$};

      \draw[line width=1.1pt] (10., .5) -- (2., 0.5+20/3) node[left]{$\frac{4}{3}$} 
      -- (2+16/3, .5) node[below]{$\frac{2}{3}$}
      -- (2., 0.5+10/3) node[left]{$\frac{2}{3}$}
      -- (2+8/3, 0.5) node[below]{$\frac{1}{3}$}
      -- (2., 0.5+5/3) node[left]{$\frac{1}{3}$};

      \draw[line width=1.1pt, dashed] (2., 0.5+5/3) -- (2+2, 0.5);

      \draw[line width=1.1pt] (2, .5+10/3) -- (2+4, 0.5) node[below]{$\frac{1}{2}$};

      \node[inner sep=0] at (10, 2) {\color{blue}$1$};
      \node[inner sep=0] at (6, 1) {\color{blue}$1$};
      \node[inner sep=0] at (7, 2) {\color{blue}$\frac{1}{2}$};
      \node[inner sep=0] at (4.7, 1) {\color{blue}$\frac{1}{2}$};

    \end{tikzpicture}
    \caption{The solution of \Cref{sharphyp-bv-fluxtrace}. 
      At the positions $x = \frac{1}{2^{n-1}}$,
      with $n \ge 1$, shock waves with speed $-\frac{3}{4}$ are originated.
      Instead at the positions $x = \frac{4}{3}\, \frac{1}{2^{n}}$,
      with $n \ge 1$, rarefaction waves are originated.
      All these waves do not intersect together inside the region $x>0$.}
    \label{fig:ex-appendice}.
  \end{figure}
\end{remark}

\section*{Acknowledgments}

All the authors are members of the Gruppo Nazionale per l'Analisi Matematica,
la Probabilità e le loro Applicazioni (GNAMPA)
of the Istituto Nazionale di Alta Matematica (INdAM).
\\
FA~is partially supported by PRIN PNRR P2022XJ9SX ``Heterogeneity on the road - modeling, analysis, control'' of the European Union - Next Generation EU.\\
AC is partially supported by the INdAM - GNAMPA Project “Matematizzare la citt\`a: mean-field games e modelli di traffico”, CUP code E5324001950001, and by Project 2022W58BJ5 (subject area: PE - Physical Sciences and Engineer-
ing) “PDEs and optimal control methods in mean field games, population dynamics and multi-agent models”.\\
 GMC has been partially
supported by the Project funded under the National Recovery and
Resilience Plan (NRRP), Mission 4 Component 2 Investment 1.4 -Call for
tender No. 3138 of 16/12/2021 of MUR funded by the
EU-NextGenerationEUoAward Number: CN000023, Concession Decree No. 1033
of 17/06/2022 adopted by the MUR, CUP: D93C22000410001, Centro
Nazionale per la Mobilit\`a Sostenibile, the MUR under the Programme
Department of Excellence Legge 232/2016 (Grant No. CUP -
D93C23000100001), and the Research Project of National Relevance
``Evolution problems involving interacting scales'' granted by the MUR
(Prin 2022, project code 2022M9BKBC, Grant No. CUP D53D23005880006).\\
MG is partial supported by the
PRIN~2022 project \emph{Modeling, Control and Games through Partial
  Differential Equations} (CUP~D53D23005620006), funded by the
European Union - Next Generation EU, by the
European Union-NextGeneration EU (National Sustainable Mobility Center
CN00000023, Italian Ministry of University and Research Decree
n.~1033-17/06/2022, Spoke 8) and by the Project funded under the
National Recovery and Resilience Plan (NRRP) of Italian Ministry of
University and Research funded by the European Union-NextGeneration
EU. Award Number: ECS\_00000037, CUP: H43C22000510001,
MUSA-Multilayered Urban Sustainability Action.

%
%

\end{document}